\documentclass[12pt]{article}
\usepackage{amsmath,amssymb,amsthm,bm}
\usepackage{enumitem}
\usepackage{color}

\newtheorem{dfe}{Definition}
[section]
\newtheorem{lem}[dfe]{Lemma}
\newtheorem{theo}[dfe]{Theorem}
\newtheorem{ex}[dfe]{Example}
\newtheorem{pro}[dfe]{Proposition}

\newtheorem{Rem}{Remark}

\makeatletter

\@addtoreset{equation}{section}
\makeatother



\newcommand{\C}{\mathbb{C}}

\newcommand{\Z}{\mathbb{Z}}
\newcommand{\N}{\mathbb{N}}
\newcommand{\F}{\mathbb{F}}
\newcommand{\Span}{\mathop{\mathrm{Span}}\nolimits}
\newcommand{\dom}{\mathop{\mathrm{dom}}\nolimits}
\newcommand{\type}{\mathop{\mathrm{type}}\nolimits}
\newcommand{\zerovec}{o}

\newcommand{\MD}{\mathrm{multideg}}
\newcommand{\LT}{\mathrm{LT}}

\begin{title}
{Multivariate $P$- and/or $Q$-polynomial association schemes}
\end{title}
\author{
    Eiichi Bannai\thanks{Faculty of Mathematics, Kyushu University (emeritus), Japan},
    Hirotake Kurihara\thanks{Department of Applied Science, Yamaguchi University, 2-16-1 Tokiwadai, Ube 755-8611, Japan}, 
    Da Zhao\thanks{Department of Applied Mathematics and Physics, Graduate School of
    Informatics, Kyoto University, Sakyo-Ku, Kyoto, 606 8501, Japan}, 
    Yan Zhu\thanks{College of Science, University of Shanghai for Science and Technology, Shanghai
    200093, China}}
\begin{document}
\maketitle

\begin{abstract}
The classification problem of $P$- and $Q$-polynomial association schemes has been one of the central problems in algebraic combinatorics. 
Generalizing the concept of $P$- and $Q$-polynomial association schemes to multivariate cases, namely to consider higher rank $P$- and $Q$-polynomial association schemes, has been tried by some authors, but it seems that so far there were neither very well-established definitions nor results. 
Very recently,  Bernard, Cramp\'{e}, d'Andecy, Vinet, and Zaimi~\cite{bi}, defined bivariate $P$-polynomial association schemes, as well as bivariate $Q$-polynomial association schemes.  
In this paper, we study these concepts and propose a new modified definition concerning a general monomial order, which is more general and more natural and also easy to handle. 
We prove that there are many interesting families of examples of multivariate $P$- and/or $Q$-polynomial association schemes. 
\end{abstract}

\tableofcontents

\section{Introduction}
\label{sec:intro}

The classification problem of $P$- and $Q$-polynomial 
association schemes has been one of the most 
important problems in algebraic combinatorics. A historical survey on this problem 
can be seen in Chapter 6 of Bannai-Bannai-Ito-Tanaka~\cite{BBIT2021} in particular. Also, Leonard~\cite{Leonard1982}, Bannai-Ito~\cite{BI1984}, Brouwer-Cohen-Neumaier~\cite{BCN1989}, Terwilliger~\cite{Terwilliger2001,Terwilliger2021}, etc., are basic references.
There have been many attempts to consider 
higher rank $P$- and $Q$-polynomial association schemes 
as well as higher rank (i.e., multivariate versions of) Askey-Wilson orthogonal polynomials.
However, until recently, there was no precise framework to define higher rank $P$- and/or $Q$-polynomial association schemes, despite many efforts to obtain multivariate Askey-Wilson orthogonal polynomials, especially at the level of orthogonal polynomials. 
Relevant works on this subject include 
Mizukawa~\cite{Mizukawa2004}, 
Mizukawa-Tanaka~\cite{MT2004}, 
Gasper-Rahman~\cite{GR2007}, 
Scarabotti~\cite{Scarabotti2011},
Iliev-Terwilliger~\cite{IT2012},
and several other papers cited in 
Bernard-Cramp\'{e}-d'Andecy-Vinet-Zaimi~\cite{bi}.
On the other hand, at the level of association schemes,   
the recent pioneering paper of Bernard et al.~\cite{bi} introduced a new very interesting framework of bivariate $P$-polynomial (and also $Q$-polynomial) association 
schemes, and did show that there are many 
interesting explicit examples of bivariate $P$-polynomial association schemes.
The present paper by the authors was strongly motivated by the paper \cite{bi}. 

The present paper has the following 
two main purposes.

\begin{enumerate}[label=$(\roman*)$]
\item We will look at the concept of multivariate 
$P$-polynomial (and $Q$-polynomial) association 
schemes and will give a more general definition 
of them, modifying the original definitions of Bernard et al.
~\cite{bi}. 
We will avoid the use of $(\alpha,\beta)$ for the type of bivariate $P$-polynomial association schemes. Which we think will make the situation less technical and more transparent.
Also, the generalization to multivariate $(>2)$ cases becomes immediate. 
We believe that this new definition is more natural and general overall.
Furthermore, there are expected to be many more examples of multivariate $P$-polynomial (and $Q$-polynomial) association 
schemes in this new sense. 

\item We will discuss many such new explicit examples of 
multivariate $P$-polynomial (and $Q$-polynomial) association 
schemes. (We hope to discuss further 
such examples in a subsequent paper.) 
\end{enumerate}

Now we will explain more details of the contents of 
this paper. 
In Section~\ref{sec:2}, we review the work of 
Bernard et al.~\cite{bi} on bivariate $P$-polynomial (and $Q$-polynomial) association 
schemes and then propose our new 
revised definitions (Definitions~\ref{df:abPpoly} and \ref{df:abQpoly}). 
Thanks to the new definition,
it can be shown that the Bose-Mesner algebra of a multivariate polynomial association scheme
is isomorphic to $\C[x_1,x_2,\ldots ,x_\ell]/I$
using a certain ideal of $\C[x_1,x_2,\ldots ,x_\ell]$.
This will be discussed in more detail in Proposition~\ref{prop:A=C[x]/I}.
We also prove 
that the important properties, namely 
Proposition 2.4, Lemma 2.5 and Proposition 2.6  
in \cite{bi} also hold for our new definition, namely 
Proposition~\ref{prop:vij_unique}, Lemma~\ref{lem:recur}, Proposition~\ref{prop:P-TFAE}, respectively. 
In Section~\ref{sec:3}, we will explain explicit examples
of bivariate $P$-polynomial association schemes in our new definition 
but do not satisfy the exact condition in the 
definition of \cite{bi}.
Furthermore, we give a new example of a bivariate $Q$-polynomial association scheme.
In Section~\ref{sec:Examplesmultpoly}, we consider more examples which are 
essentially multivariate $P$-polynomial and/or $Q$-polynomial 
association schemes.
Here, our examples are based on the 
examples of Gelfand pairs described in 
the paper of Ceccherini-Silberstein, Scarabotti and Tolli~\cite{CST2006}.
More specifically, we obtain Theorem~\ref{thm:GHS} on 
the extensions of association schemes
(or the generalized Hamming schemes in the sense 
of \cite{CST2006}, namely $F \wr \mathfrak{S}_n/H \wr \mathfrak{S}_n$). 
Also, we obtain Theorem~\ref{thm:GJS} on the generalized Johnson association 
schemes $F \wr \mathfrak{S}_n/(H \wr \mathfrak{S}_h\times F \wr \mathfrak{S}_{n-h})$ in the sense of \cite{CST2006}.
In Section~\ref{sec:Furthercomments}, we will give further comments and 
speculations mostly 
without proof, hoping that more details will be 
discussed in a subsequent paper.

\section{Definition of multivariate $P$-polynomial and/or $Q$-polynomial 
association schemes}
\label{sec:2}

\subsection{Association schemes}
\label{sec:AS}
In this subsection, 
we begin by recalling the concept of association schemes.
The reader is referred to Bannai-Bannai-Ito-Tanaka~\cite{BBIT2021} and Bannai-Ito~\cite{BI1984} for details.
Let $X$ and $\mathcal{I}$ be finite sets
and let $\mathcal{R}$ be a surjective map from $X\times X$ to $\mathcal{I}$.
For $i\in \mathcal{I}$, we put $R_i=\mathcal{R}^{-1}(i)$,
i.e., $R_i=\{(x,y)\in X\times X\mid \mathcal{R}(x,y)=i\}$.
Let $M_X(\C)$ be the $\C$-algebra of complex matrices with rows and columns indexed by $X$.
The \emph{adjacency matrix} $A_i$ of $i\in \mathcal{I}$
is defined to be the matrix in $M_X(\C)$ whose $(x,y)$ entries are
\[
(A_i)_{xy}=
\begin{cases}
    1 & \text{if $\mathcal{R}(x,y)=i$,}\\   
    0 & \text{otherwise.}
\end{cases}    
\]
It is obvious that
\begin{enumerate}[label=$(\mathrm{A}\arabic*)$]
    \item $\sum_{i\in \mathcal{I}} A_i =J_X$, where $J_X$ is the all-ones matrix of $M_X(\C)$. \label{AS:Hadamard}
\end{enumerate}
A pair $\mathfrak{X}=(X,\{R_i\}_{i\in \mathcal{I}})$ (or simply $(X,\mathcal{R})$)
is called a \emph{commutative association scheme} if
$\mathfrak{X}$ satisfies the following conditions:
\begin{enumerate}[label=$(\mathrm{A}\arabic*)$]
    \setcounter{enumi}{1}
    \item there exists $i_0 \in \mathcal{I}$ such that $A_{i_0}=I_X$, where $I_X$ is the identity matrix of $M_X(\C)$; \label{AS:I}
    \item for each $i\in \mathcal{I}$, there exists $i'\in \mathcal{I}$ such that $A_i^T=A_{i'}$, where $A_i^T$ denotes the transpose of $A_i$; \label{AS:transpose}
    \item for each $i,j\in \mathcal{I}$,
    \[
    A_i A_j = \sum_{k\in \mathcal{I}} p^k_{ij} A_k   
    \]
    holds.
    The constant $p^k_{ij}$ is called the \emph{intersection number}; \label{AS:pijk}
    \item for $i,j,k\in \mathcal{I}$, $p^k_{ij}=p^k_{ji}$ holds, i.e., $A_i A_j = A_j A_i$ holds.\label{AS:commutative}
\end{enumerate}
If the cardinality $|\mathcal{I}|$ of $\mathcal{I}$ is equal to $d+1$,
then $\mathfrak{X}$ is called of class $d$.
Moreover, if an association scheme $\mathfrak{X}=(X,\{R_i\}_{i\in \mathcal{I}})$ satisfies 
\begin{enumerate}[label=$(\mathrm{A}\arabic*)$]
    \setcounter{enumi}{5}
\item for each $i\in \mathcal{I}$, $i=i'$ holds, \label{AS:symmetric}
\end{enumerate}
then $\mathfrak{X}$ is called \emph{symmetric}.

We also use the notation $\mathfrak{X}=(X,\{A_i\}_{i\in \mathcal{I}})$
to denote association schemes with the adjacency matrices $\{A_i\}_{i\in \mathcal{I}}$.

\begin{ex}
\label{ex:Gelfandpair}
Let $G$ be a finite group acting transitively on a finite set $X$.
For a fixed element $x_0\in X$, Let $K$ be the stabilizer of $x_0$.
It is known that $X$ can be regarded as the coset space $G/K$
and the space $L(X)=\{f\colon X\to \C\}$ of $\C$-valued functions on $X$
is a $G$-space with the action $g f(x):=f(g^{-1}x)$ for $g\in G$ and $f\in L(X)$.
The pair $(G,K)$ is called a \emph{Gelfand pair}
if the decomposition $L(X)=\bigoplus^d_{j=0}V_j$ into irreducible $G$-modules is multiplicity-free. 
For a Gelfand pair $(G,K)$,
let $X=G/K$, $\mathcal{I}=K\backslash G /K$
and $\mathcal{R}\colon X\times X \to \mathcal{I}$ by $\mathcal{R}(g_1K, g_2K)=K g_1^{-1} g_2 K$.
Then $\mathfrak{X}=(X,\mathcal{R})$ becomes a commutative association scheme.

The $K$-orbit decomposition of $X$ corresponds to $\mathcal{I}$
because $x,y\in X$ are in the same $K$-orbit if and only if $\mathcal{R}(x_0,x)=\mathcal{R}(x_0,y)$.
Thus, we denote the decomposition
$X=\bigsqcup _{i\in \mathcal{I}} \Lambda_i$ 
of the $K$-orbits of $X$, and
we have 
\begin{equation}
    \label{eq:lamda=R}
    R_i= \{(g x_0,g x)\mid x\in \Lambda_i,\  g\in G\}.
\end{equation}
The right-hand side of \eqref{eq:lamda=R} is denoted by $\widetilde{\Lambda_i}$.
\end{ex}

Let $\mathfrak{A}=\Span_\C \{ A_i \}_{i\in \mathcal{I}}$.
By \ref{AS:I} and \ref{AS:pijk}, $\mathfrak{A}$ becomes a subalgebra of $M_X(\C)$.
The algebra $\mathfrak{A}$ is called the \emph{Bose-Mesner algebra} of $\mathfrak{X}$.
By \ref{AS:Hadamard}, $\{A_i\}_{i\in \mathcal{I}}$ is a basis of $\mathfrak{A}$
and we have $\dim_{\C} \mathfrak{A}=d+1$
if $\mathfrak{X}$ is of class $d$.

By \ref{AS:commutative}, $\mathfrak{A}$ is semisimple.
This implies that $\mathfrak{A}$ has another basis  
$\{E_j\}_{j\in \mathcal{J}}$ consisting of the primitive idempotents of $\mathfrak{A}$,
where $\mathcal{J}$ is a finite set
and there exists $j_0\in \mathcal{J}$ such that $E_{j_0}=\frac{1}{|X|}J_X$.
Since $\{A_i\}_{i\in \mathcal{I}}$ and $\{E_j\}_{j\in \mathcal{J}}$ are bases of $\mathfrak{A}$,
$|\mathcal{I}|=|\mathcal{J}|$ holds.
By \ref{AS:Hadamard}, $\mathfrak{A}$ is closed under entrywise multiplication,
which product is denoted by $\circ$ and called the \emph{Hadamard product}.
Then for $i,j,k\in \mathcal{J}$,
there exists a real number (in fact, nonnegative number) $q^k_{ij}$
such that $(|X| E_i)\circ (|X|E_j)=\sum_{k\in \mathcal{J}}q^k_{ij}|X|E_k$,
and $q^k_{ij}$ are called the \emph{Krein numbers} of $\mathfrak{X}$.

The entries of the \emph{first eigenmatrix} $P:=(P_i(j))_{j\in \mathcal{J}, i\in \mathcal{I}}$
and the \emph{second eigenmatrix} $Q:=(Q_j(i))_{i\in \mathcal{I}, j\in \mathcal{J}}$
are defined by
\[
A_i = \sum_{j\in \mathcal{J}}P_i(j)E_j \  
\text{and}\  
|X| E_j = \sum_{i\in \mathcal{I}}Q_j(i)A_i,
\]
respectively.

A symmetric association scheme $\mathfrak{X}=(X,\{R_i\}_{i\in \mathcal{I}})$ of class $d$
is called \emph{$P$-polynomial}
if it satisfies the following conditions:
$\mathcal{I}=\{0,1,\ldots ,d\}$
and there exists a univariate polynomial $v_i$ of degree $i$ such that $A_i=v_i(A_1)$
for each $i\in \{0,1,\ldots ,d\}$. 
Similarly, a symmetric association scheme $\mathfrak{X}=(X,\{R_i\}_{i\in \mathcal{I}})$ of class $d$
is called \emph{$Q$-polynomial}
if it satisfies the following conditions:
$\mathcal{J}=\{0,1,\ldots ,d\}$
and there exists a univariate polynomial $v^\ast_j$ of degree $j$ such that $|X|E_j=v^\ast _j(|X| E_1)$
(under the Hadamard product)
for each $j\in \{0,1,\ldots ,d\}$. 
The following condition is well known as an equivalent condition of the property of $P$-polynomial: for $i\in \{0,1,\ldots ,d\}$, the three-term recurrence formula 
\[
    A_1 A_i = p^{i-1}_{1i} A_{i-1} +p^{i}_{1i} A_i +p^{i+1}_{1i} A_{i+1}
\]
holds.
Note that $p^{-1}_{10}A_{-1}$ and $p^{d+1}_{1d}A_{d+1}$ are regarded as zero.
Similarly, an equivalent condition of the property of $Q$-polynomial
is the following: for $i\in \{0,1,\ldots ,d\}$, the three-term recurrence formula 
\[
    (|X| E_1) \circ (|X| E_i) = q^{i-1}_{1i}|X| E_{i-1} +q^{i}_{1i} |X| E_i +q^{i+1}_{1i} |X| E_{i+1}
\]
holds. 
Note that $q^{-1}_{10}|X|E_{-1}$ and $q^{d+1}_{1d}|X|E_{d+1}$ are regarded as zero.

\subsection{Monomial orders and Gr\"obner bases}
\label{sec:monomial_orders}

In this subsection, we explain the fundamentals of monomial orders and Gr\"obner bases. 
For further details, please refer to Cox-Little-O'Shea~\cite{Cox}.

The following notation for $\N^\ell:=\{(n_1,n_2,\ldots ,n_\ell) \mid \text{$n_i$ are nonnegative integers}\}$
will be used in this paper:
\begin{itemize}
    \item $o := (0,0,\ldots ,0)\in \N^\ell$;
    \item for $i=1,2,\ldots ,\ell$, $\epsilon_i\in \N^\ell$ denotes the $\ell$-tuple in which the $i$-th entry is $1$ and the remaining entries are $0$;
    \item for $\alpha=(n_1,n_2, \ldots ,n_\ell), \beta=(m_1,m_2, \ldots ,m_\ell)\in \N^\ell$,
    let $\alpha+\beta$ be $(n_1+m_1, n_2+m_2, \ldots , n_\ell +m_\ell)$,
    and when $n_i\ge m_i$ ($i=1,2,\ldots , \ell$), let $\alpha-\beta$ be $(n_1-m_1, n_2-m_2, \ldots , n_\ell -m_\ell)$;
    \item for $\alpha=(n_1,n_2, \ldots ,n_\ell)\in \N^\ell$, we write $\sum^\ell_{i=1}n_i$ by $|\alpha|$.
\end{itemize}

\begin{dfe}
    \label{df:monomialorder}
    A \emph{monomial order} $\le$ on $\C [x_1, x_2, \ldots , x_\ell]$
    is a relation on the set of monomials $x_1^{n_1}x_2^{n_2}\cdots x_\ell^{n_\ell}$ satisfying:
    \begin{enumerate}[label=$(\roman*)$]
        \item $\le$ is a total order;
        \item for monomials $u,v,w$, if $u\le v$, then $wu \le wv$;
        \item $\le$ is a well-ordering, i.e., any nonempty subset of the set of monomials
        has a minimum element under $\le$.
    \end{enumerate}
\end{dfe}
For $\alpha=(n_1,n_2,\ldots , n_\ell)\in \N^\ell$
and $\bm{x}=(x_1,x_2,\ldots , x_\ell)$,
we write the monomial $x_1^{n_1}x_2^{n_2}\cdots x_\ell^{n_\ell}$
by $\bm{x}^\alpha$.
Then $\alpha$ is called the \emph{multidegree} of $\bm{x}^\alpha$.
We shall use the same symbol $\le$ to denote an order of pairs in $\N^\ell$.
Since $\bm{x}^\alpha \cdot \bm{x}^\beta =\bm{x}^{\alpha+\beta}$,
the condition (ii) of Definition~\ref{df:monomialorder} is equivalent to
\begin{equation}
    \label{eq:a+c_b+c}
    \alpha\le \beta\  
    \  \Longrightarrow  \   
    \alpha + \gamma \le \beta + \gamma
\end{equation}
for $\alpha,\beta,\gamma\in \N^\ell$.
In particular, we have,
for $i=1,2,\ldots ,\ell$,
\begin{equation}
    \label{eq:slide}
    \alpha \le \beta
    \ \Longrightarrow \ 
    \alpha +\epsilon_i \le \beta +\epsilon_i.
\end{equation}

\begin{Rem}
    \label{Rem:a>0}
    By \cite[p.73]{Cox},
    the condition (iii) of Definition~\ref{df:monomialorder}
    is equivalent to
    \begin{equation}
        \label{eq:zero_is_min}
        \alpha \ge o\  \text{for any}\  \alpha\in \N^\ell.
    \end{equation}
    Moreover, by \eqref{eq:a+c_b+c} and \eqref{eq:zero_is_min},
    for $\alpha=(n_1,n_2,\ldots , n_\ell), \beta=(m_1,m_2,\ldots , m_\ell)\in \N^\ell$
    with $n_i\ge m_i$ ($i=1,2,\ldots ,\ell$),
    we have
    \begin{equation}
        \label{eq:b_le_a}
        \alpha \ge \beta.
    \end{equation}
\end{Rem}

\begin{ex}
Let $\alpha=(n_1,n_2, \ldots ,n_\ell), \beta=(m_1,m_2, \ldots ,m_\ell)\in \N^\ell$.
\begin{enumerate}[label=$(\roman*)$]
\item 
We define $\alpha \le_{\mathrm{lex}} \beta$ if the leftmost nonzero entry of 
$\alpha-\beta\in \Z^\ell$ is negative.
This  $\le_{\mathrm{lex}}$ is called the \emph{lexicographic}
(or \emph{lex}) order.

\item 
We define $\alpha \le_{\mathrm{grlex}} \beta$ if
\[
    |\alpha|<|\beta|\  \text{or}\   (|\alpha|=|\beta|\   \text{and}\  \alpha \le_{\mathrm{lex}} \beta).    
\]
This $\le_{\mathrm{grlex}}$ is called the \emph{graded lexicographic} (or \emph{grlex}) order.

\end{enumerate}
\end{ex}

Fix a monomial order $\le$ on $\N^\ell$.
For each nonzero polynomial $f\in \C[\bm{x}]$,
the \emph{multidegree} $\MD(f)$ and the \emph{leading term}  $\LT(f)$ of
$f  =\sum_{\alpha \in \N^\ell} c_\alpha \bm{x}^\alpha$ are defined as follows:
\[\MD(f) := \max\{ \alpha \in \N^\ell \mid   c_\alpha \neq 0   \} \;\;\text{and}\;\; \LT(f) := c_{\MD(f)} \bm{x}^{\MD(f)},\]
where the maximum is taken with respect to $\le$.

Let $I\subset \C[\bm{x}]$ be an ideal.
A subset $\mathcal{G}=\{g_1,\dots,g_m\} \subset I$
is called a \emph{Gr\"obner basis} of $I$ with respect to $\le$
if the following monomial ideals are equal:
\[ \langle  \LT(g_1), \LT(g_2),\ldots , \LT(g_m) \rangle =  \langle \LT(f) \mid f\in I  \rangle .\]
Equivalently, $\mathcal{G}$ is a Gr\"obner basis of $I$ if and only if
the leading term of any element of $I$ is divisible by one of the $\LT(g_i)$.
It is well known that $\mathcal{G}$ is a generating set of $I$
if $\mathcal{G}$ is a Gr\"obner basis of $I$.
We set the following two subsets of $\N^\ell$: $\MD(I):=\{\MD (f) \mid f\in I\setminus \{0\}\}$ and
$\MD(\mathcal{G}):=\{\MD (g) \mid g\in \mathcal{G}\}$.
To rephrase the definition of Gr\"obner bases in terms of multidegrees,
for a Gr\"obner basis $\mathcal{G}$ of $I$,
one can see
\begin{equation}
    \label{eq:MD(I)}
    \MD(I)=\{\alpha+\beta  \mid \alpha \in \MD(\mathcal{G}),\  \beta \in \N^\ell\}.
\end{equation}

In general, for a generator set $\mathcal{F}=\{f_1,f_2,\ldots ,f_m\}$ of an ideal $I$,
the quotient or remainder of a polynomial $f\in \C[\bm{x}]$
by $\mathcal{F}$ may not be uniquely determined.
Nevertheless, if the generator set of $I$ is
a Gr\"obner basis,
then the quotient and remainder of $f$ are uniquely determined.
More precisely, it is as follows.
Let $I\subset \C [\bm{x}]$ be an ideal and $\mathcal{G}=\{g_1, g_2, \ldots ,g_m\}$
be a Gr\"obner basis of $I$ with respect to $\le$.
Then given $f\in \C[\bm{x}]$, there exist unique $r\in \C[\bm{x}]$ and $g\in I$
such that $f=g +r$ and 
no term of $r$ is divisible by $\LT(g_i)$ for any $i=1,2,\ldots ,m$
(see~\cite[p.83]{Cox}).
The polynomial $r$ is called the \emph{remainder} of $f$ on division by $\mathcal{G}$.
We write $r=\overline{f}^\mathcal{G}$.
Since $f\equiv \overline{f}^\mathcal{G} \pmod{I}$,
we can regard $\overline{f}^\mathcal{G}$ as representatives of the quotient ring $\C[\bm{x}]/I$ of $\C[\bm{x}]$ by $I$.
From this fact, the following proposition is obtained.


\begin{pro}(see~\cite[p.248]{Cox})
\label{pro:quotient_ring}
Fix a monomial ordering on $\N^\ell$ and let $I$ be an ideal of $\C[\bm{x}]$.
The following hold:
\begin{enumerate}[label=$(\roman*)$]
\item the elements of $\{\bm{x}^\alpha \mid \alpha \notin \MD(I)\}$
are linearly independent modulo $I$;
\item $\C[\bm{x}]/I$ is isomorphic as a $\C$-vector space to $\Span \{\bm{x}^\alpha \mid \alpha \notin \MD(I)\}$.
Moreover, $\C[\bm{x}]/I$ has the product structure $[f][g]=[fg]$,
namely, $\C[\bm{x}]/I$ is a commutative algebra.
\end{enumerate}
\end{pro}

\subsection{Bivariate polynomial 
association schemes of Bernard-Cramp\'{e}-d'Andecy-Vinet-Zaimi}
\label{sec:Bernard}

Bernard et al.~\cite{bi} introduced the concept of $(a,b)$-compatible
bivariate polynomials and $(a,b)$-compatible domains
in order to define bivariate $P$-polynomial and/or $Q$-polynomial association schemes
\footnote{In~\cite{bi}, they used the symbol ``$(\alpha,\beta)$-compatible''.
In this paper, $(\alpha,\beta)$ will be denoted by $(a,b)$ to distinguish it from the symbols $\alpha,\beta\in \N^\ell$ in Subsection~\ref{sec:monomial_orders}.}.
We would like to briefly review the definitions here.

For real numbers $0\le a \le 1$ and $0\le b <1$,
a partial order $\preceq_{(a,b)}$ on $\N^2$ is defined by
\[
(m,n)\preceq_{(a,b)}(i,j)
\Longleftrightarrow
\begin{cases}
    m + a n \le i + a j & \text{and}\\
    b m + n \le b i +j.
\end{cases}
\]
Note that if $(m,n)\preceq_{(a,b)}(i,j)$, then $(n,m)\le_{\mathrm{grlex}} (j,i)$
\footnote{
The grlex order defined in Subsection~\ref{sec:monomial_orders} satisfies $(0,1)\le_{\mathrm{grlex}} (1,0)$,
but in \cite{bi} it requires that
$(1,0)\le_{\mathrm{grlex}} (0,1)$
by permuting the first and second variables.
}.

A bivariate polynomial $v(x, y)$ is called
\emph{$(a,b)$-compatible} of degree $(i, j)$
if $v(x, y)$ is written as
\[
v(x,y) = \sum_{(m,n)\preceq_{(a,b)}(i,j)}
c_{mn} x^m y^n
\]
with $c_{ij}\neq 0$.
Also, a domain $\mathcal{D}\subset \N^2$ is called
\emph{$(a,b)$-compatible} if for any $(i,j)\in \mathcal{D}$,
one has
\[
    (m,n)\preceq_{(a,b)}(i,j)
    \ 
    \Longrightarrow
    (m,n)\in \mathcal{D}.
\]

\begin{dfe}[Definition 2.3 of \cite{bi}]
    \label{df:bPpoly}
    Let $\mathcal{D}\subset \N^2$.
    A symmetric association scheme $\mathfrak{X}=(X,\{A_i\}_{i\in \mathcal{I}})$ is called \emph{bivariate $P$-polynomial}
    of type $(a,b)$ on the domain $\mathcal{D}$ 
    if the following two conditions are satisfied:
    \begin{enumerate}[label=$(\roman*)$]
        \item there exists a relabeling of the adjacency matrices of $\mathfrak{X}$:
        \[
        \{A_i\}_{i\in \mathcal{I}} = \{A_{mn}\}_{(m,n)\in \mathcal{D}},   
        \]
        such that, for $(i,j)\in \mathcal{D}$,
        \[
            A_{ij}=v_{ij}(A_{10},A_{01}),    
        \]
        where $v_{ij}(x,y)$ is a $(a,b)$-compatible bivariate polynomial of degree $(i,j)$;
        \item $\mathcal{D}$ is $(a,b)$-compatible.
    \end{enumerate}
\end{dfe}

\begin{dfe}[Definition 4.1 of \cite{bi}]
    \label{df:bQpoly}
    Let $\mathcal{D}^\ast\subset \N^2$.
    A symmetric association scheme $\mathfrak{X}=(X,\mathcal{R})$ with the primitive idempotents $\{E_j\}_{j\in \mathcal{J}}$ is called \emph{bivariate $Q$-polynomial}
    of type $(a,b)$ on the domain $\mathcal{D}^\ast$ 
    if the following two conditions are satisfied:
    \begin{enumerate}[label=$(\roman*)$]
        \item there exists a relabeling of the primitive idempotents of $\mathfrak{X}$:
        \[
        \{E_j\}_{j\in \mathcal{J}} = \{E_{mn}\}_{(m,n)\in \mathcal{D}^\ast},   
        \]
        such that, for $(i,j)\in \mathcal{D}^\ast$,
        \[
            |X| E_{ij}=v^\ast_{ij}(|X|E_{10},|X|E_{01})
            \ \text{(under the Hadamard product)},    
        \]
        where $v^\ast_{ij}(x,y)$ is a $(a,b)$-compatible bivariate polynomial of degree $(i,j)$;
        \item $\mathcal{D}^\ast$ is $(a,b)$-compatible.
    \end{enumerate}
\end{dfe}

\subsection{Definition of multivariate polynomial association schemes}
\label{sec:Def}

In this subsection, 
the definitions of multivariate $P$-polynomial and/or $Q$-polynomial 
association schemes are introduced,
following the definition of bivariate $P$-polynomial and/or $Q$-polynomial in Subsection~\ref{sec:Bernard}.
Note that the new definitions are generalizations from bivariate to multivariate, from type $(a,b)$ to monomial order $\le$ and from symmetric to commutative, with respect to the definitions in Subsection~\ref{sec:Bernard}.
We also show that analogies of the results in \cite{bi}
hold for the definitions of multivariate $P$-polynomial and/or $Q$-polynomial association schemes as well.


\begin{dfe}
    \label{df:abPpoly}
    Let $\mathcal{D}\subset \N^\ell$
    having $\epsilon_1,\epsilon_2,\ldots ,\epsilon_\ell$
    and $\le$ be a monomial order on $\N^\ell$.
    A commutative association scheme $\mathfrak{X}=(X,\mathcal{R})$ is called \emph{$\ell$-variate $P$-polynomial}
    on the domain $\mathcal{D}$ with respect to $\le$
    if the following three conditions are satisfied:
    \begin{enumerate}[label=$(\roman*)$]
        \item If $(n_1,n_2,\ldots ,n_\ell)\in \mathcal{D}$
        and $0\le m_i \leq n_i$ for $i=1,2,\ldots ,\ell$,
        then $(m_1,m_2,\ldots ,m_\ell)\in \mathcal{D}$;
        \item there exists a relabeling of the adjacency matrices of $\mathfrak{X}$:
        \[
        \{A_i\}_{i\in \mathcal{I}} = \{A_\alpha \}_{\alpha\in \mathcal{D}},   
        \]
        such that, for $\alpha\in \mathcal{D}$,
        \begin{equation}
            \label{eq:A=vij}
            A_\alpha=v_\alpha(A_{\epsilon_1},A_{\epsilon_2},\ldots ,A_{\epsilon_\ell}),    
        \end{equation}
        where $v_\alpha(\bm{x})$ is an $\ell$-variate polynomial
        of multidegree $\alpha$ with respect to $\le$
        and all monomials $\bm{x}^\beta$ in $v_\alpha(\bm{x})$ satisfy $\beta \in \mathcal{D}$;
        \item for $i=1,2,\ldots ,\ell$ and $\alpha=(n_1,n_2,\ldots,n_\ell)\in \mathcal{D}$,
        the product
        $A_{\epsilon_i}\cdot A_{\epsilon_1}^{n_1}A_{\epsilon_2}^{n_2}\cdots A_{\epsilon_\ell}^{n_\ell}$
        is a linear combination of
        \[
         \{A_{\epsilon_1}^{m_1}A_{\epsilon_2}^{m_2}\cdots A_{\epsilon_\ell}^{m_\ell} \mid \beta =(m_1,m_2,\ldots ,m_\ell)\in \mathcal{D},\  \beta \le  \alpha +\epsilon_i\}.   
        \]
    \end{enumerate}

\end{dfe}

Hereafter, we use notation $\bm{A}$ to denote $(A_{\epsilon_1},A_{\epsilon_2},\ldots ,A_{\epsilon_\ell})$.
Also, for $\alpha=(n_1,n_2,\ldots , n_\ell)\in \N^\ell$,
we write the monomial $A_{\epsilon_1}^{n_1} A_{\epsilon_2}^{n_2} \cdots A_{\epsilon_\ell}^{n_\ell}$
by $\bm{A}^\alpha$.

The requirement in Condition (iii) of Definition~\ref{df:abPpoly} may seem odd at first sight.
In general, $\alpha+\epsilon_i$ is not always in $\mathcal{D}$,
thus $A_{\epsilon_i}\cdot \bm{A}^\alpha
=\bm{A}^{\alpha+\epsilon_i}$ may not be represented by an $\ell$-variate polynomial of multidegree $\alpha+\epsilon_i$ on $\mathcal{D}$.
For this reason, we control the behavior of $A_{\epsilon_i}\cdot \bm{A}^\alpha$ via the condition (iii) of Definition~\ref{df:abPpoly}.

\begin{Rem}
\label{rem:biP1}
By (i) of Definition~\ref{df:abPpoly},
$\mathcal{D}$ must contain $o=(0,\ldots ,0)$.
Moreover, $A_{o}$ coincides with the identity matrix $I_X$.
\end{Rem}

\begin{Rem}
\label{rem:biP2}
For any commutative association scheme
$\mathfrak{X}=(X,\mathcal{R})$ of class $d$ with $\mathcal{I}=\{0,1,2,\ldots ,d\}$,
let $\mathcal{D}=\{o,\epsilon_1,\epsilon_2,\ldots ,\epsilon_d\}$
and put $A_0=A_o$ and $A_i=A_{\epsilon_i}$ for $i=1,2,\ldots ,d$.
Then $\mathfrak{X}$ is a $d$-variate $P$-polynomial association scheme
on $\mathcal{D}$ with respect to the graded lexicographic order $\le_{\mathrm{grlex}}$.
Therefore, we usually consider the ``essential'' variate
for $\mathfrak{X}$, i.e., we consider
$\ell =\min \{\ell' \mid \text{$\mathfrak{X}$ is $\ell'$-variate $P$-polynomial}\} $.
\end{Rem}

\begin{Rem}
\label{rem:biP3}
We do not know whether
all bivariate $P$-polynomial association schemes of type $(a,b)$
in the sense of \cite{bi} become
bivariate $P$-polynomial association schemes
in the sense of Definition~\ref{df:abPpoly} or not.
However, we can show that all examples in \cite{bi} are
bivariate $P$-polynomial association schemes
in the sense of Definition~\ref{df:abPpoly} with respect to the graded lexicographic order $\le_{\mathrm{grlex}}$.
For direct product of $P$-polynomial association schemes,
symmetrizations of association schemes and
nonbinary Johnson association schemes,
we will prove this fact in a more general setting in Section~\ref{sec:DirectproductAS},
\ref{sec:GHS} and \ref{sec:GJS}, respectively.
For association schemes obtained from attenuated spaces,
we will prove this fact for any parameter in Theorem~\ref{thm:attenuated}.
For a generalized 24-cell (resp. association schemes based on isotropic spaces),
using (3.27) and (3.28) (resp. (3.49)) in \cite{bi},
it is possible to prove this fact similarly to Theorem~\ref{thm:attenuated}.
\end{Rem}

Multivariate $Q$-polynomial association schemes can be defined
as in Definition~\ref{df:abPpoly}.
Note that Remarks \ref{rem:biP1}, \ref{rem:biP2} and \ref{rem:biP3} for the following multivariate $Q$-polynomial versions also hold.

\begin{dfe}
    \label{df:abQpoly}
    Let $\mathcal{D}^\ast \subset \N^\ell$
    having $\epsilon_1,\epsilon_2,\ldots ,\epsilon_\ell$
    and $\le$ be a monomial order on $\N^\ell$.
    A commutative association scheme $\mathfrak{X}=(X,\mathcal{R})$
    with the primitive idempotents $\{E_j\}_{j\in \mathcal{J}}$ is called \emph{$\ell$-variate $Q$-polynomial}
    on the domain $\mathcal{D}^\ast$ with respect to $\le$
    if the following three conditions are satisfied:
    \begin{enumerate}[label=$(\roman*)$]
        \item if $(n_1,n_2,\ldots ,n_\ell)\in \mathcal{D}^\ast$
        and $0\le m_i \leq n_i$ for $i=1,2,\ldots ,\ell$,
        then $(m_1,m_2,\ldots ,m_\ell)\in \mathcal{D}^\ast$;
        \item there exists a relabeling of the adjacency matrices:
        \[
        \{E_j\}_{j\in \mathcal{J}} = \{E_\alpha \}_{\alpha\in \mathcal{D}^\ast}, 
        \]
        such that, for $\alpha\in \mathcal{D}^\ast$,
        \[
            |X| E_\alpha=v^\ast_\alpha(|X| E_{\epsilon_1},|X| E_{\epsilon_2},\ldots ,|X| E_{\epsilon_\ell})
            \ \text{(under the Hadamard product),}
        \]
        where $v^\ast_\alpha(\bm{x})$ is an $\ell$-variate polynomial
        of multidegree $\alpha$ with respect to $\le$
        and all monomials $\bm{x}^\beta$ in $v^\ast_\alpha(\bm{x})$ satisfy $\beta \in \mathcal{D}^\ast$;
        \item for $i=1,2,\ldots ,\ell$ and $\alpha=(n_1,n_2,\ldots,n_\ell)\in \mathcal{D}^\ast$,
        the product
        $E_{\epsilon_i}\circ E_{\epsilon_1}^{\circ n_1}\circ E_{\epsilon_2}^{\circ n_2}\circ \cdots  \circ E_{\epsilon_\ell}^{\circ n_\ell}$
        is a linear combination of
        \[
            \{E_{\epsilon_1}^{\circ m_1}\circ E_{\epsilon_2}^{\circ m_2}\circ \cdots  \circ E_{\epsilon_\ell}^{\circ m_\ell} \mid \beta=(m_1,m_2,\ldots ,m_\ell)\in \mathcal{D}^\ast,\  \beta\le  \alpha+\epsilon_i\}.   
        \]
    \end{enumerate}

\end{dfe}

In the following, we will see various properties of multivariate $P$-polynomial association schemes.
First, we see that the condition (iii) of Definition~\ref{df:abPpoly} can be generalized as follows.
\begin{lem}
\label{lem:generalized_iii}
Let $\mathfrak{X}$ be an $\ell$-variate $P$-polynomial association scheme on $\mathcal{D}$.
For each $\alpha \in \N^\ell$, $\bm{A}^\alpha$
is a linear combination of $\{\bm{A}^\beta \mid \beta \in \mathcal{D},\  \beta \le  \alpha\}$.   
\end{lem}

\begin{proof}
We use induction on $\le$ to check this.
If $\alpha=o$, then it obviously holds.
Now assume that $\alpha=(n_1,n_2,\ldots ,n_\ell) > o$.
Then there exists $i$ such that $n_i\ge 1$.
This implies that $\alpha - \epsilon_i \in \N^\ell$
and $\alpha - \epsilon_i < \alpha$ by \eqref{eq:b_le_a}.
If $\alpha - \epsilon_i\in \mathcal{D}$, then
by (iii) of Definition~\ref{df:abPpoly},
$\bm{A}^\alpha = A_{\epsilon_i} \bm{A}^{\alpha-\epsilon_i}$
is combination of $\{\bm{A}^\beta \mid \beta \in \mathcal{D},\  \beta \le  \alpha\}$.
Otherwise, by the induction hypothesis, 
$\bm{A}^{\alpha-\epsilon_i}$ is written as
\[
    \bm{A}^{\alpha-\epsilon_i}= \sum_{\substack{\beta\in \mathcal{D}\\ \beta < \alpha-\epsilon_i}} c_{\beta} \bm{A}^\beta.   
\]
Thus, we have
\[
    \bm{A}^{\alpha}= \sum_{\substack{\beta\in \mathcal{D}\\ \beta < \alpha-\epsilon_i}} c_{\beta} \bm{A}^{\beta+\epsilon_i}.   
\]
By \eqref{eq:a+c_b+c}, $\beta+\epsilon_i < \alpha$ holds for any $\beta < \alpha-\epsilon_i$.
Hence, by the induction hypothesis again, $\bm{A}^{\beta +\epsilon_i}$ is a linear combination of $\{\bm{A}^\gamma \mid \gamma \in \mathcal{D},\  \gamma \le \beta+\epsilon_i\}$.
By transitivity of $\le$, we have $\gamma \le \beta+\epsilon_i<\alpha$.
Therefore, the desired result follows.
\end{proof}


\begin{lem}
\label{lem:BMalg}
If $\mathfrak{X}$ is an $\ell$-variate $P$-polynomial association scheme on $\mathcal{D}$,
then 
\begin{align}
    \label{eq:geneset}
    \{\bm{A}^\alpha
    \mid \alpha \in \mathcal{D} \}
\end{align}
is a basis of the Bose-Mesner algebra of $\mathfrak{X}$.
\end{lem}

\begin{proof}
By Definition~\ref{df:abPpoly} (ii),
for each $\alpha \in \mathcal{D}$,
the polynomial $v_\alpha(\bm{x})$ of equation \eqref{eq:A=vij}
is a linear combination of the monomials $\bm{x}^\beta$
with $\beta \in \mathcal{D}$.
Thus, each $A_\alpha$ can be written as a linear combination of \eqref{eq:geneset}.
Since the cardinality of $\mathcal{D}$ is equal to
the dimension of the Bose-Mesner algebra,
the generating set \eqref{eq:geneset} is linearly independent.
Then the desired result follows.
\end{proof}

By (iii) of Definition~\ref{df:abPpoly},
for $\alpha\in \mathcal{D}$ and $i=1,2,\ldots ,\ell$ with $\alpha+\epsilon_i \notin \mathcal{D}$,
there exists a polynomial 
\[
    w_{\alpha+\epsilon_i}(\bm{x}):=
    \bm{x}^{\alpha+\epsilon_i} + \sum_{\substack{\beta\in \mathcal{D}\\ \beta < \alpha+\epsilon_i}} c_{\beta} \bm{x}^\beta
\]
of multidegree $\alpha+\epsilon_i$ in $\C [\bm{x}]$
such that $w_{\alpha+\epsilon_i}(\bm{A})=0$.
If there exist $\alpha , \alpha' \in \mathcal{D}$
and $i,i'=1,2,\ldots ,\ell$
such that $\alpha +\epsilon_i = \alpha' + \epsilon_{i'}$
and $\alpha +\epsilon_i, \alpha' + \epsilon_{i'}\notin \mathcal{D}$,
then $w_{\alpha+\epsilon_i}(\bm{x})=w_{\alpha'+\epsilon_{i'}}(\bm{x})$
by Lemma~\ref{lem:BMalg}.
Let $I$ be the ideal of $\C [\bm{x}]$
generated by 
\begin{equation}
    \label{eq:ideal}
    \mathcal{G}:=
    \{w_{\alpha+\epsilon_i}(\bm{x}) \mid \alpha\in \mathcal{D},\  i=1,2,\ldots ,\ell,\ \alpha + \epsilon_i \notin \mathcal{D}\}.
\end{equation}

\begin{lem}
\label{lem:Dout}
Let $\mathfrak{X}$ be an $\ell$-variate $P$-polynomial association scheme on $\mathcal{D}$
with respect to a monomial order $\le$.
For $\alpha\in \N^\ell \setminus \mathcal{D}$,
there exist $\beta_0 \in \MD(\mathcal{G})$ 
and $\gamma_0 \in \N^\ell$
such that $\alpha = \beta_0+\gamma_0$.
\end{lem}

\begin{proof}
Let $\mathcal{C}_\alpha
:=
\{
\beta\in \N^\ell \setminus \mathcal{D}
\mid
\text{there exists $\gamma \in \N^\ell$
such that $\alpha = \beta +\gamma$}
\}$.
Note that $\alpha \in \mathcal{C}_\alpha$
and $\mathcal{C}_\alpha \neq \emptyset$.
Since $\le$ is a well-ordering, there exists 
the minimum element $\beta_0$ in $\mathcal{C}_\alpha$.
By Remark~\ref{rem:biP1}, we have $o \notin \mathcal{C}_\alpha$.
Thus, $\beta_0 \neq o$.
Then there exists $i=1,2,\ldots ,\ell$ such that
$\beta_0 - \epsilon_i \in \N^\ell$.
By the minimality of $\beta_0 \in \mathcal{C}_\alpha$,  it follows that $\beta_0 - \epsilon_i\in \mathcal{D}$.
Therefore, we have $w_{\beta_0}=w_{(\beta_0-\epsilon_i)+\epsilon_i} \in \mathcal{G}$
and $\beta_0 \in \MD(\mathcal{G})$.
Moreover, by the definition of $\mathcal{C}_\alpha$, there exists $\gamma_0 \in \N^\ell$
such that $\alpha = \beta_0 +\gamma_0$.
\end{proof}

\begin{pro}
    \label{prop:A=C[x]/I}
    Let $\mathfrak{X}$ be an $\ell$-variate $P$-polynomial association scheme on $\mathcal{D}$
    with respect to a monomial order $\le$.
    Then the followings hold:
    \begin{enumerate}[label=$(\roman*)$]
        \item $\mathcal{G}$ is a Gr\"obner basis of $I$;
        \item $\MD(I)=\N^\ell \setminus \mathcal{D}$ holds;
        \item The Bose-Mesner algebra $\mathfrak{A}$ of $\mathfrak{X}$
        is isomorphic to $\C [\bm{x}]/I$ as algebra.
    \end{enumerate}            
\end{pro}

\begin{proof}
(i)
Let $f\in I$ with $\MD(f)=\alpha$.
Then there exist $\{q_g\}_{g\in \mathcal{G}} \subset \C [\bm{x}]$
such that $f=\sum_{g\in \mathcal{G}} q_g g$.
Assume $\alpha\in \mathcal{D}$.
Then $f$ is written as
\[
f=\sum_{\substack{\beta\in \N^\ell\\ \beta \le \alpha}} c_{\beta} \bm{x}^\beta,
\]
where $c_{\alpha}\neq 0$.
We now calculate $f(\bm{A})$ in two different ways.
Since $g(\bm{A})=0$ for $g\in \mathcal{G}$,
one can see $f(\bm{A})=0$.
On the other hand, by Lemma~\ref{lem:generalized_iii},
we have
\[
f(\bm{A})=c_\alpha \bm{A}^\alpha +\sum_{\substack{\beta\in \mathcal{D}\\ \beta < \alpha}} c'_{\beta} \bm{A}^\beta,
\]
where each $c'_\beta$ is a linear combination of $c_\gamma$ with $\gamma < \alpha$.
By Lemma~\ref{lem:BMalg}, the set $\{\bm{A}^\beta \mid \beta \in \mathcal{D},\  \beta \le  \alpha\}$
is linearly independent.
This implies that $c_\alpha =0$.
This is a contradiction. Hence, we have $\alpha \notin \mathcal{D}$.
By Lemma~\ref{lem:Dout}, there exist $\beta_0 \in \MD(\mathcal{G})$ and $\gamma_0 \in \N^\ell$
such that $\alpha = \beta_0+\gamma_0$.
This means $\LT(f)=c_\alpha \bm{x}^{\gamma_0} \LT(g)$ for some $g\in \mathcal{G}$.
Therefore, $\mathcal{G}$ is a Gr\"obner basis of $I$.

(ii)
Firstly, we show $\N^\ell \setminus \mathcal{D}\subset \MD(I)$.
Take $\alpha \in \N^\ell \setminus \mathcal{D}$.
By Lemma~\ref{lem:Dout}, there exist $\beta_0 \in \MD(\mathcal{G})$ and $\gamma_0 \in \N^\ell$
such that $\alpha = \beta_0+\gamma_0$.
Since $\mathcal{G}$ is a Gr\"obner basis of $I$ and by \eqref{eq:MD(I)},
we have $\alpha \in \MD(I)$.

Next, we show $\MD(I) \subset \N^\ell \setminus \mathcal{D}$.
Take $\alpha \in \MD(I)$.
Since $\mathcal{G}$ is a Gr\"obner basis of $I$ and by \eqref{eq:MD(I)},
there exist $\beta \in \MD(\mathcal{G})$ and $\gamma\in \N^\ell$ such that
$\alpha=\gamma + \beta$.
Assume $\alpha \in \mathcal{D}$.
Since the $i$-th entry of $\alpha -\beta$ is equal to the $i$-th entry of $\gamma$
for $i=1,2,\ldots ,\ell$,
all entries of $\alpha -\beta$ are nonnegative.
Then by (i) of Definition~\ref{df:abPpoly}, we have $\beta \in \mathcal{D}$.
This is a contradiction for $g\in \mathcal{G}$.
Hence, we have $\alpha \notin \mathcal{D}$.
This implies $\mathcal{D}\subset \N^\ell \setminus \MD(I)$.

(iii)
Consider the homomorphism $\Phi \colon \C [\bm{x}]/I \to \mathfrak{A}$
defined by $\Phi([f]):=f(\bm{A})$.
Since any $q\in I$ satisfies $q(\bm{A})=0$,  $\Phi$ is well-defined.
By Proposition~\ref{pro:quotient_ring} and (ii) of Proposition~\ref{prop:A=C[x]/I},
the set $\{\bm{x}^\alpha \mid \alpha \in \mathcal{D}\}$ is a basis of $\C [\bm{x}]/I$.
On the other hand, by Lemma~\ref{lem:BMalg}, the set $\{\bm{A}^\alpha \mid \alpha \in \mathcal{D}\}$
is a basis of $\mathfrak{A}$.
Hence, $\Phi$ is bijective.
Moreover, since 
\[
\Phi([f][g])=\Phi([fg])=fg(\bm{A})=f(\bm{A})g(\bm{A})=\Phi([f])\Phi([g])
\]
holds,
the homomorphism $\Phi$ is an isomorphism as algebra.
\end{proof}

For the new definition (Definition~\ref{df:abPpoly}),
we also have similar results of Proposition~2.4,
Lemma~2.5 and
Proposition~2.6
in \cite{bi}.

\begin{pro}[cf.~Proposition~2.4 of \cite{bi}]
    \label{prop:vij_unique}
    Let $\mathfrak{X}$ be an $\ell$-variate $P$-polynomial association scheme
    on $\mathcal{D}$ with respect to a monomial order $\le$.
    Then, for all $\alpha\in \mathcal{D}$,
    the polynomial $v_\alpha(\bm{x})$ satisfying equation
    \eqref{eq:A=vij} is unique.
\end{pro}
\begin{proof}
Suppose now that there is another
polynomial $v'_\alpha(\bm{x}) \neq v_\alpha(\bm{x})$ of multidegree $\alpha$
such that $A_\alpha = v'_\alpha(\bm{A})$.
Since the monomials $\bm{x}^\beta$ are linearly independent by Lemma~\ref{lem:BMalg},
this implies that there is a linear relation between the matrices
$\bm{A}^\beta$
for $\beta\in \mathcal{D}$,
which contradicts their linear independence.
\end{proof}

\begin{Rem}
    By Proposition~\ref{prop:vij_unique},
    $v_{\epsilon_i}(\bm{A})=A_{\epsilon_i}$
    and $v_{\epsilon_i}(\bm{x})=x_i$ hold.
\end{Rem}

\begin{lem}[cf.~Lemma~2.5 of \cite{bi}]
\label{lem:recur}
Let $\{A_\alpha\}_{\alpha\in \mathcal{D}}$ be the adjacency matrices of
an $\ell$-variate $P$-polynomial association scheme $\mathfrak{X}$ with respect to a monomial order $\le$ on $\mathcal{D}$.
For $i=1,2,\ldots ,\ell$ and $\alpha\in \mathcal{D}$,
we have
\begin{eqnarray}
    \label{eq:xvi-1j}
    A_{\epsilon_i} A_{\alpha} = 
    \sum_{\substack{\beta\in \mathcal{D}, \\ \beta\le \alpha+\epsilon_i}}
    p_{\epsilon_i,\alpha}^\beta A_\beta.
\end{eqnarray}
Moreover, if $\alpha+\epsilon_i \in \mathcal{D}$, then $p_{\epsilon_i,\alpha}^{\alpha + \epsilon_i} \neq 0$ holds.
\end{lem}

\begin{proof}
Let $\{v_\alpha\}_{\alpha\in \mathcal{D}}$ be the $\ell$-variate polynomials associated to $\mathfrak{X}$.
Fix $\alpha \in \mathcal{D}$.
Firstly, we prove
\begin{equation}
    \label{eq:spanvij} 
    \Span\{A_\beta \mid \beta\in \mathcal{D},\   \beta\le \alpha\}
    =
    \Span\{\bm{A}^\beta \mid \beta\in \mathcal{D},\   \beta\le \alpha\}.
\end{equation}
Take $A_\beta$ in (LHS) of \eqref{eq:spanvij}.
By (ii) of Definition~\ref{df:abPpoly}, we have $A_\beta=v_\beta(\bm{A})$.
Since $v_\beta(\bm{x})$ is a polynomial of multidegree $\beta$
on $\mathcal{D}$, 
we have
\[
    v_{\beta}(\bm{A}) =
    \sum_{\substack{\gamma\in \mathcal{D}, \\ \gamma\le \beta}}
    c_\gamma \bm{A}^\gamma.
\]
Since $\le$ is a total order, $\gamma\le \alpha$ holds.
Therefore, $v_\beta(\bm{A})$ belongs to (RHS) of \eqref{eq:spanvij}.
This implies that (LHS) $\subset $ (RHS) holds.
Moreover, since
\[
\{A_\beta \mid \beta\in \mathcal{D},\   \beta\le \alpha \}
\]
is linearly independent in the Bose-Mesner algebra of $\mathfrak{X}$,
we have
\[
    \dim \Span\{v_\beta(\bm{A}) \mid \beta\in \mathcal{D},\   \beta\le \alpha \}
    =
    |\{ \beta\in \mathcal{D} \mid   \beta\le \alpha \}|.  
\]
This implies that \eqref{eq:spanvij} holds.

Put $v_{\alpha}(\bm{x})=\sum_{\substack{\beta\in \mathcal{D}, \\ \beta\le \alpha}}
c_\beta \bm{x}^\beta$ with $c_{\alpha} \neq 0$.
Then it follows
\[
    A_{\epsilon_i} A_{\alpha}
    =
    A_{\epsilon_i} v_{\alpha}(\bm{A})
    =
    A_{\epsilon_i} \sum_{\substack{\beta\in \mathcal{D}, \\ \beta\le \alpha}}
    c_\beta \bm{A}^\beta
    =
    \sum_{\substack{\beta\in \mathcal{D}, \\ \beta \le \alpha}}
    c_\beta \bm{A}^{\beta+\epsilon_i}.
\]
By (iii) of Definition~\ref{df:abPpoly},
$\bm{A}^{\beta+\epsilon_i}$ is a linear combination of $\bm{A}^\gamma$ 
for $\gamma \in \mathcal{D}$ with $\gamma\le \beta+\epsilon_i$.
Moreover, by $\beta+\epsilon_i \le \alpha +\epsilon_i$ and \eqref{eq:spanvij},
we have
\[
    A_{\epsilon_i} A_{\alpha}=
    \sum_{\substack{\beta\in \mathcal{D}, \\ \beta \le \alpha+\epsilon_i}}
    c'_\beta \bm{A}^{\beta}
    =
    \sum_{\substack{\beta\in \mathcal{D}, \\ \beta \le \alpha+\epsilon_i}}
    c''_\beta A_{\beta}
\]
for some $c'_\beta , c''_\beta \in \C$.
Comparing the definition of the intersection numbers and knowing that the matrices
$A_\beta$ are independent,
$c''_\beta=p^\beta_{\epsilon_i,\alpha}$ holds.

Furthermore, assuming that $\alpha + \epsilon_i \in \mathcal{D}$, 
following the above transformation, 
taking care to $c_\alpha \neq 0$, 
we can see $c'_{\alpha + \epsilon_i}= c_\alpha$ and $c''_{\alpha + \epsilon_i} \neq 0$.
This means that $p^{\alpha+\epsilon_i}_{\epsilon_i,\alpha}$ is nonzero.
\end{proof}

\begin{pro}[cf.~Proposition~2.6 of \cite{bi}]
\label{prop:P-TFAE}
Let $\mathcal{D}\subset \N^\ell$ having $\epsilon_1,\epsilon_2,\ldots ,\epsilon_\ell$ and
$\mathfrak{X} = (X,\{A_\alpha\}_{\alpha\in \mathcal{D}})$
be a commutative association scheme.
Then the statements (i) and (ii) are equivalent:
\begin{enumerate}[label=$(\roman*)$]
    \item $\mathfrak{X}$ is
    an $\ell$-variate $P$-polynomial association scheme 
    on $\mathcal{D}$ with respect to a monomial order $\le$;
    \item
    the condition (i) of Definition~\ref{df:abPpoly} holds for $\mathcal{D}$
    and the intersection numbers satisfy,
    for each $i=1,2,\ldots ,\ell$ and each $\alpha \in \mathcal{D}$,
    $p^\beta_{\epsilon_i, \alpha} \neq 0$ for $\beta\in \mathcal{D}$ implies $\beta \le \alpha+\epsilon_i$. 
    Moreover, if $\alpha+\epsilon_i\in \mathcal{D}$, then $p^{\alpha+\epsilon_i}_{\epsilon_i, \alpha} \neq 0$ holds.
\end{enumerate}
\end{pro}
\begin{proof}
(i) $\Longrightarrow$ (ii):
from Lemma~\ref{lem:recur}, 
(ii) follows.

(ii) $\Longrightarrow$ (i):
by Remark~\ref{rem:biP1}, $\mathcal{D}$ contains $o$.
We use induction on $\le$ to check that
(ii) and (iii) of Definition~\ref{df:abPpoly} hold.
It is immediate for $\alpha=o$.
Now assume that $\alpha > o$.
Then there exists $i$ such that $n_i\ge 1$.
We note that $\alpha-\epsilon_i \in \mathcal{D}$ by (i) of Definition~\ref{df:abPpoly}.
By the assumption (ii), we have
\begin{equation}
    \label{eq:induction1}
    A_{\epsilon_i} A_{\alpha-\epsilon_i}
    =
    p^{\alpha}_{\epsilon_i,\alpha-\epsilon_i}A_{\alpha}
    +\sum_{\substack{\beta\in \mathcal{D}, \\ \beta < \alpha}}
    p^{\beta}_{\epsilon_i,\alpha-\epsilon_i} A_{\beta}. 
\end{equation}
For the adjacency matrices
$\{A_{\beta} \mid \beta\in \mathcal{D},\ \beta< \alpha  \}$
in the right-hand side of \eqref{eq:induction1},
by the induction hypothesis of (ii) of Definition~\ref{df:abPpoly},
these are expressed as polynomials of the monomials $\{\bm{A}^{\beta} \mid \beta\in \mathcal{D},\ \beta< \alpha \}$.
Also, by the induction hypothesis of (ii) of Definition~\ref{df:abPpoly},
$A_{\alpha-\epsilon_i}$ is expressed as a polynomial
$\sum_{\substack{\beta\in \mathcal{D}, \\ \beta\le \alpha -\epsilon_i}}    
c_{\beta} \bm{A}^{\beta}$ with $c_{\alpha -\epsilon_i} \neq 0$.
Moreover, by the induction hypothesis of (iii) of Definition~\ref{df:abPpoly},
we have
\[
    A_{\epsilon_i} A_{\alpha-\epsilon_i}
    =
    \sum_{\substack{\beta\in \mathcal{D}, \\ \beta\le \alpha -\epsilon_i}}    
    c_{\beta} \bm{A}^{\beta + \epsilon_i}
    =
\sum_{\substack{\beta\in \mathcal{D}, \\ \beta< \alpha -\epsilon_i}}    
c_{\beta} \bm{A}^{\beta + \epsilon_i} + c_{\alpha -\epsilon_i} \bm{A}^{\alpha}
=
\sum_{\substack{\beta\in \mathcal{D}, \\ \beta< \alpha }}    
c'_{\beta} \bm{A}^{\beta} + c_{\alpha -\epsilon_i} \bm{A}^{\alpha}
\]
for some $c'_\beta \in \C$.
Therefore,
since $p^\alpha_{\epsilon_i,\alpha-\epsilon_i}\neq 0$ because of $\alpha = (\alpha -\epsilon_i)+\epsilon_i \in \mathcal{D}$, 
we have that $A_\alpha$ is expressed as a polynomial
$v_\alpha(\bm{x})$ evaluated in $\{ A_{\epsilon_j} \}^\ell_{j=1}$.
Moreover, the fact that $v_{\alpha}(\bm{x})$ is 
a polynomial of multidegree $\alpha$
on $\mathcal{D}$
follows now easily from the transitivity of $\le$.

Finally, we prove that (iii) of Definition~\ref{df:abPpoly} holds.
By the induction hypothesis of (ii) of Definition~\ref{df:abPpoly}
and the above argument,
for $\beta \in \mathcal{D}$ with $\beta \le \alpha$,
$A_\beta$ is expressed as a 
polynomial of multidegree $\beta$ evaluated in $\{ A_{\epsilon_j} \}^\ell_{j=1}$.
This implies that \eqref{eq:spanvij} holds.
Hence, $\bm{A}^\alpha$ is expressed as a linear combination of 
$A_\beta$
for $\beta \in \mathcal{D}$ with $\beta \le \alpha$.
Thus, by the assumption (ii),
we have
\[
A_{\epsilon_i} \bm{A}^\alpha=
 \sum_{\substack{\beta\in \mathcal{D}, \\ \beta\le \alpha}} c_\beta A_{\epsilon_i} A_\beta
=
\sum_{\substack{\beta\in \mathcal{D}, \\ \beta\le \alpha}} c_\beta 
\sum_{\substack{\gamma\in \mathcal{D}, \\ \gamma\le \beta+\epsilon_i}} p^\gamma_{\epsilon_i,\beta} A_\gamma
\]
for some $c_\beta \in \C$.
From the transitivity of $\le$, we have $\gamma \le \alpha +\epsilon_i$
for $\gamma$ on the right-hand side in the above equation.
Then, $A_{\epsilon_i} \bm{A}^\alpha$ is expressed as a linear combination of $A_\gamma$
for $\gamma \in \mathcal{D}$ with $\gamma \le \alpha +\epsilon_i$.
Therefore, using \eqref{eq:spanvij} again,
we have that
$A_{\epsilon_i} \bm{A}^\alpha$ is expressed as a linear combination of $\bm{A}^\gamma$
for $\gamma \in \mathcal{D}$ with $\gamma \le \alpha +\epsilon_i$.
\end{proof}

\begin{Rem}
    If $\mathfrak{X}$ is symmetric,
    then the intersection numbers satisfy the following
    symmetry property: $p^\beta_{\alpha ,\gamma}=0$ $\Longleftrightarrow$ $p^\gamma_{\alpha,\beta}=0$.
    Therefore, in this case,
    (ii) in Proposition~\ref{prop:P-TFAE} are replaced by the following:
    $p^\beta_{\epsilon_i, \alpha} \neq 0$ for $\beta\in \mathcal{D}$ implies $\beta \le \alpha+\epsilon_i$ and $\alpha \le \beta +\epsilon_i$. 
    Moreover, $p^{\alpha+\epsilon_i}_{\epsilon_i, \alpha} \neq 0$ holds if $\alpha+\epsilon_i\in \mathcal{D}$
    and
    $p^{\alpha}_{\epsilon_i, \alpha-\epsilon_i} \neq 0$ holds if $\alpha-\epsilon_i\in \mathcal{D}$.
\end{Rem}

We have similar result of Proposition~\ref{prop:P-TFAE}
as follows:
\begin{pro}
    \label{prop:Q-TFAE}
    Let $\mathcal{D}^\ast\subset \N^\ell$  having $\epsilon_1,\epsilon_2,\ldots ,\epsilon_\ell$ and
    $\mathfrak{X}$
    be a commutative association scheme
    with the primitive idempotents $\{E_\alpha\}_{\alpha\in \mathcal{D}^\ast}$
    indexed by $\mathcal{D}^\ast$.
    The statements (i) and (ii) are equivalent:
    \begin{enumerate}[label=$(\roman*)$]
        \item $\mathfrak{X}$ is
        an $\ell$-variate $Q$-polynomial association scheme 
        on $\mathcal{D}^\ast$ with respect to $\le$;
        \item
        the condition (i) of Definition~\ref{df:abQpoly} holds for $\mathcal{D}^\ast$
        and the Krein numbers satisfy,
        for each $i=1,2,\ldots ,\ell$ and each $\alpha \in \mathcal{D}^\ast$,
        $q^\beta_{\epsilon_i, \alpha} \neq 0$ for $\beta\in \mathcal{D}^\ast$ implies $\beta \le \alpha+\epsilon_i$. 
        Moreover, if $\alpha+\epsilon_i\in \mathcal{D}^\ast$, then $q^{\alpha+\epsilon_i}_{\epsilon_i, \alpha} \neq 0$ holds.
    \end{enumerate}
\end{pro}

\section{New examples of bivariate polynomial association schemes}
\label{sec:3}

\subsection{Association schemes obtained from attenuated spaces}
\label{sec:attenuated}

Let us recall the definition of the association schemes obtained from attenuated spaces.
For a prime power $q$, a positive integer $n$ and a nonnegative integer $l$,
fix an $l$-dimensional subspace $W$ of the $(n+l)$-dimensional vector space $\F_q^{n+l}$
over the finite field $\F_q$ of $q$ elements.
The corresponding \emph{attenuated space} associated with $\F_q^{n+l}$ and $W$
is the collection of all subspaces of $\F_q^{n+l}$ intersecting trivially with $W$.
For a positive integer $m$ with $m\le n$,
let $X$ be the set of $m$-dimensional subspaces of the attenuated space associated with $\F_q^{n+l}$ and $W$.
Let 
\[
\mathcal{D}:=\{(i,j)\mid 0\le i \le \min\{m,n-m\},\  0\le j\le \min\{m-i,l\}\}
\]
and $\mathcal{R}\colon X \times X \to \mathcal{D}$ is defined by
$\mathcal{R}(V,V')=(i,j)$ if
\[
\dim V/W\cap V'/W=m-i \  \text{and}\  \dim V\cap V'=(m-i)-j,
\]
where $V/W$ stand for $(V+W)/W$ simply.
Then $\mathfrak{X}=(X,\mathcal{R})$ is a symmetric association scheme.
For details of the association schemes obtained from attenuated spaces, 
see Bernard et al.~\cite{bi}, 
Wang-Guo-Li~\cite{WGL2010}
or Kurihara~\cite{kurihara2013}
\footnote{Note that the roles of $i$ and $j$ are reversed in \cite{bi} and \cite{kurihara2013,WGL2010}.
In this paper, we adapt the notation of \cite{kurihara2013,WGL2010}.}.

Bernard et al.~\cite{bi} proved that in the case of $l\ge m$,
$\mathfrak{X}$ becomes bivariate $P$-polynomial
of type $(1,0)$ on the domain $\mathcal{D}$.
However, in the case $l<m$, the framework in their definition did not treat $\mathfrak{X}$
as a bivariate $P$-polynomial association scheme.
On the other hand, using Definition~\ref{df:abPpoly},
we can show that
$X$ becomes bivariate $P$-polynomial
on the domain $\mathcal{D}$
with respect to $\le_\mathrm{grlex}$
even if $l<m$.

\begin{theo}
\label{thm:attenuated}
Let $\mathfrak{X}$ be the association scheme obtained from attenuated spaces as above.
Then $\mathfrak{X}$ is a bivariate $P$-polynomial association scheme
on the domain $\mathcal{D}$ with respect to $\le_\mathrm{grlex}$.
\end{theo}
\begin{proof}
In order to prove this, we use Proposition~\ref{prop:P-TFAE}.
Firstly we check (i) of Definition~\ref{df:abPpoly}.
Take $(i,j)\in \mathcal{D}$ and $0\le i'\le i$ and $0\le j'\le j$.
Then:
\begin{itemize}
    \item since $i' \le i \le \min\{m, n-m\}$, we have $0 \le i' \le \min\{m, n-m\}$;
    \item since $j' \le j \le \min\{m-i,l\}\le \min\{m-i',l\}$, we have $0\le j' \le \min\{m-i',l\}$.
\end{itemize}
This implies $(i',j')\in \mathcal{D}$.

In \cite{bi}, the linear expansions of $A_{10}A_{ij}$ and $A_{01}A_{ij}$
are given as (3.58) and (3.59), respectively (see also (13) and (14) in \cite{WGL2010}).
We describe (3.58) and (3.59) in \cite{bi} according to our notation:
for $(i,j)\in \mathcal{D}$, we have
\begin{align}
    A_{10}A_{ij}=&
q^{2i+j+l-1}[m-i-j+1]_q [n-m-i+1]_q A_{i-1,j}
+[i+1]_q^2 q^j A_{i+1,j} \notag\\
&+[m-i-j+1]_q [i]_q (q^l-q^{j-1})q^{i+j} A_{i,j-1}
+[j+1]_q [i]_q q^{i+j+1} A_{i,j+1} \notag\\
&+[i+1]_q^2 (q^l-q^{j-1})A_{i+1,j-1}
+[j+1]_q [n-m-i+1]_q q^{2i+l-1} A_{i-1,j+1} \notag\\
&+[i]_q([n-m-i]_q q^{l+1+i}+[m-i-j]_q q^{i+2j+1}+[j]_q (q^l-q^{j-1})q^{i+1} +[i]_q (q-1)q^l)A_{ij}, \label{eq:attenuated10}\\
A_{01}A_{ij}=&
(q^l-q^{j-1})[m-i-j+1]_q q^{i+j-1}A_{i,j-1}
+[j+1]_q q^{i+j} A_{i,j+1} \notag\\
&+((q^l-1)[i+j]_q -[j]_q q^{i+j-1}+
(q-1)[m-i-j]_q [j]_q q^{i+j} )A_{ij}, \label{eq:attenuated01}
\end{align}
where $[n]_q:=(q^n-1)/(q-1)$ are $q$-numbers.
Note that 
in the right-hand sides of \eqref{eq:attenuated10}
and \eqref{eq:attenuated01},
the terms whose indices do not belong to $\mathcal{D}$ are
regarded as not appearing.
The indices appearing in the right-hand sides of \eqref{eq:attenuated10} are
\[
\{(i-1,j),(i,j-1),(i-1,j+1),(i,j),(i+1,j-1),(i,j+1),(i+1,j)\} \cap \mathcal{D}. 
\]
Thus, these indices are less than or equal to $(i,j)+(1,0)=(i+1,j)$
with respect to the graded lexicographic order $\le_\mathrm{grlex}$.
Moreover, if $(i+1,j)\in \mathcal{D}$, then $p^{(i+1,j)}_{(1,0)(i,j)}=[i+1]_q^2 q^j \neq 0$.
Also, the indices appearing on the right-hand sides of \eqref{eq:attenuated01} are
\[
\{(i,j-1),(i,j),(i,j+1)\} \cap \mathcal{D}. 
\]
Thus, these indices are less than or equal to $(i,j)+(0,1)=(i,j+1)$
with respect to $\le_\mathrm{grlex}$.
Moreover, if $(i,j+1)\in \mathcal{D}$, then $p^{(i,j+1)}_{(0,1)(i,j)}=[j+1]_q q^{i+j} \neq 0$.
\end{proof}

\subsection{Dodecahedron}
\label{sec:Dodecahedron}

It is known that the 1-skeleton of a dodecahedron is a univariate $P$-polynomial association scheme
(i.e., a distance-regular graph) of class 5 and not a univariate $Q$-polynomial association scheme.
Let $\mathfrak{X}=(X,\{R_i\}^5_{i=0})$ be the association scheme obtained from a dodecahedron.
The first and second eigenmatrices of $\mathfrak{X}$ are given by
\[
P=
\left(\begin{array}{rrrrrr}
1 & 3 & 6 & 6 & 3 & 1 \\
1 & \sqrt{5} & 2 & -2 & -\sqrt{5} & -1 \\
1 & 1 & -2 & -2 & 1 & 1 \\
1 & 0 & -3 & 3 & 0 & -1 \\
1 & -2 & 1 & 1 & -2 & 1 \\
1 & -\sqrt{5} & 2 & -2 & \sqrt{5} & -1
\end{array}\right),
Q=
\left(\begin{array}{rrrrrr}
1 & 3 & 5 & 4 & 4 & 3 \\
1 & \sqrt{5} & \frac{5}{3} & 0 & -\frac{8}{3} & -\sqrt{5} \\
1 & 1 & -\frac{5}{3} & -2 & \frac{2}{3} & 1 \\
1 & -1 & -\frac{5}{3} & 2 & \frac{2}{3} & -1 \\
1 & -\sqrt{5} & \frac{5}{3} & 0 & -\frac{8}{3} & \sqrt{5} \\
1 & -3 & 5 & -4 & 4 & -3
\end{array}\right),
\]
respectively.
Let us define $\mathcal{D}^\ast:=\{(0,0),(0,1),(0,2),(0,3),(1,0),(1,1)\}$ and 
\[
E_{00}:=20 E_0,\ 
E_{01}:=20 E_1,\ 
E_{02}:=20 E_2,\ 
E_{10}:=20 E_3,\ 
E_{11}:=20 E_4,\ 
E_{03}:=20 E_5.
\]
Note that $|X|=20$.

\begin{theo}
$\mathfrak{X}$ is a bivariate $Q$-polynomial association scheme
on $\mathcal{D}^\ast$ with respect to $\le_{\mathrm{grlex}}$.
\end{theo}

\begin{proof}
Obviously, $\mathcal{D}^\ast$ satisfies (i) of Definition~\ref{df:abQpoly}.
The matrices $L^\ast_{01}$ and $L^\ast_{10}$, 
of entries $(L^\ast_{01})_{kl,ij}=q^{kl}_{01,ij}$ and
$(L^\ast_{10})_{kl,ij}=q^{kl}_{10,ij}$, are given by
\[
L^\ast_{01}=
\left(\begin{array}{rrrrrr}
0 & 3 & 0 & 0 & 0 & 0 \\
1 & 0 & 2 & 0 & 0 & 0 \\
0 & \frac{6}{5} & 0 & \frac{8}{5} & 0 & \frac{1}{5} \\
0 & 0 & 2 & 0 & 1 & 0 \\
0 & 0 & 0 & 1 & 0 & 2 \\
0 & 0 & \frac{1}{3} & 0 & \frac{8}{3} & 0
\end{array}\right)
\ \text{and}\ 
L^\ast_{10}=
\left(\begin{array}{rrrrrr}
0 & 0 & 0 & 4 & 0 & 0 \\
0 & 0 & \frac{8}{3} & 0 & \frac{4}{3} & 0 \\
0 & \frac{8}{5} & 0 & \frac{4}{5} & 0 & \frac{8}{5} \\
1 & 0 & 1 & 0 & 2 & 0 \\
0 & 1 & 0 & 2 & 0 & 1 \\
0 & 0 & \frac{8}{3} & 0 & \frac{4}{3} & 0
\end{array}\right),
\]
where rows and columns are indexed by $00,01,02,10,11,03$.
This implies, for example, 
$E_{01}\circ E_{02}= 2E_{01}+2E_{10}+\frac{1}{3}E_{03}$
and
$E_{01}\circ E_{11}= E_{10}+\frac{8}{3}E_{03}$.
Therefore, we can check that $\mathfrak{X}$ satisfies the condition (ii) of Proposition~\ref{prop:Q-TFAE}
with respect to $\le_{\mathrm{grlex}}$.
\end{proof}

\begin{Rem}
By the recurrence relations obtained from the columns of $L^\ast_{01}$ or $L^\ast_{10}$,
we can check that these matrices satisfy
\begin{align*}
E_{02}&=\frac{5}{6}E_{01}\circ E_{01}-\frac{5}{2}E_{00},\\
E_{03}&=\frac{5}{2}E_{01}\circ E_{01} \circ E_{01} -\frac{27}{2} E_{01} -6 E_{10},\\
E_{11}&=E_{10}\circ E_{01}-\frac{4}{3}E_{01}\circ E_{01}+4 E_{00}.
\end{align*}
Thus, we obtain the associated polynomials of the bivariate $Q$-polynomial association scheme $\mathfrak{X}$ as follows:
$v^\ast_{02}(x,y):=\frac{5}{6}y^2-\frac{5}{2}$,
$v^\ast_{03}(x,y):=\frac{5}{2}y^3 -\frac{27}{2} y -6x$ and
$v^\ast_{11}(x,y):=xy-\frac{4}{3}y^2+4$.
\end{Rem}

\section{New examples of multivariate polynomial association schemes}
\label{sec:Examplesmultpoly}

\subsection{Direct product of association schemes}
\label{sec:DirectproductAS}
Let $\mathfrak{X}^{(k)}=(X^{(k)}, \{A^{(k)}_i\}^{d_k}_{i=0})$
be commutative association schemes of class $d_k$ for $k=1,2,\ldots ,\ell$.
The \emph{direct product} of $\{\mathfrak{X}^{(k)}\}^\ell_{k=1}$ is
the association scheme defined by the Kronecker product of the adjacency matrices of $\mathfrak{X}^{(k)}$:
\[
A_{(n_1,n_2,\ldots ,n_\ell)} :=
A^{(1)}_{n_1}\otimes A^{(2)}_{n_2} \otimes \cdots \otimes A^{(\ell)}_{n_\ell}
\]
for
\[ 
(n_1,n_2,\ldots ,n_\ell)\in 
\mathcal{D}:=\{0,1,\ldots , d_1\}\times
\{0,1,\ldots , d_2\}\times \cdots \times 
\{0,1,\ldots , d_\ell\}.
\]
This association scheme is denoted by $\bigotimes^{\ell}_{k=1}\mathfrak{X}^{(k)}$.
\begin{theo}
    The followings hold:
    \begin{enumerate}[label=$(\roman*)$]
    \item if $\{\mathfrak{X}^{(k)}\}^\ell_{k=1}$ are $P$-polynomial,
    then $\bigotimes^{\ell}_{k=1}\mathfrak{X}^{(k)}$ is an $\ell$-variate $P$-polynomial association scheme
    on $\mathcal{D}$ with respect to any monomial order $\le$;
    \item if $\{\mathfrak{X}^{(k)}\}^\ell_{k=1}$ are $Q$-polynomial,
    then $\bigotimes^{\ell}_{k=1}\mathfrak{X}^{(k)}$ is an $\ell$-variate $Q$-polynomial association scheme
    on $\mathcal{D}$ with respect to any monomial order $\le$.
    \end{enumerate}
\end{theo}
\begin{proof}
(i) 
Obviously, $\mathcal{D}$ satisfies (i) of Definition~\ref{df:abPpoly}.
Let $\{v^{(k)}_i\}^{d_k}_{i=0}$ be the associated polynomials of a $P$-polynomial association scheme $\mathfrak{X}^{(k)}$. 
Fix $\alpha=(n_1,n_2,\ldots ,n_\ell)\in \mathcal{D}$.
Since $\{\mathfrak{X}^{(k)}\}^\ell_{k=1}$ are $P$-polynomial,
we have
\[
A_{\alpha}
=
v_{\alpha} (A_{\epsilon_1},A_{\epsilon_2}, \ldots A_{\epsilon_\ell}),
\]
where $v_{\alpha}(\bm{x})=\prod^\ell_{k=1} v^{(k)}_{n_k}(x_k)$.
Then all monomials $x_1^{m_1}x_2^{m_2}\cdots x_\ell^{m_\ell}$ in $v_{\alpha}(\bm{x})$
satisfy $m_k \le n_k$ for $k=1,2,\ldots ,\ell$.
Put $\beta=(m_1,m_2,\ldots ,m_\ell)$.
By \eqref{eq:b_le_a},
for any monomial order $\le$, we have $\beta\le \alpha$.
This implies that the multidegree of  $v_{\alpha}(\bm{x})$ coincides with $\alpha$.
Furthermore, since $(A^{(k)}_1)^{d_k+1}$ is expressed as a linear combination of $\{(A^{(k)}_1)^i\}^{d_k}_{i=1}$
for $k=1,2,\ldots ,\ell$,
$\mathfrak{X}$ satisfies (iii) of Definition~\ref{df:abPpoly}.
Therefore, the desired result follows.

(ii) Let $\{E^{(k)}_j\}^{d_k}_{j=0}$ be
the primitive idempotents of $\mathfrak{X}^{(k)}$
and $\{v^{\ast (k)}_j\}^{d_k}_{j=0}$ be the associated polynomials of a $Q$-polynomial association scheme $\mathfrak{X}^{(k)}$. 
By Martin~\cite{Martin1999}, the primitive idempotents of $\bigotimes^{\ell}_{k=1}\mathfrak{X}^{(k)}$ are
\[
E_{\alpha} :=
E^{(1)}_{n_1}\otimes E^{(2)}_{n_2} \otimes \cdots \otimes E^{(\ell)}_{n_\ell}
\]
for $\alpha=(n_1,n_2,\ldots ,n_\ell)\in \mathcal{D}$.
Since $\{\mathfrak{X}^{(k)}\}^\ell_{k=1}$ are $Q$-polynomial,
we have
\[
N E_{\alpha}
=
v^\ast_{\alpha} (N E_{\epsilon_1},N E_{\epsilon_2}, \ldots ,N E_{\epsilon_\ell}),
\]
where 
$N=\prod^\ell_{k=1}|X^{(k)}|$ and
$v^\ast_{\alpha}(\bm{x})=\prod^\ell_{k=1} v^{\ast (k)}_{n_k}(x_k)$.
Similarly to (i), we can show that $\bigotimes^{\ell}_{k=1}\mathfrak{X}^{(k)}$ is $\ell$-variate $Q$-polynomial.
\end{proof}


To conclude this subsection, we give an example where the direct product of association schemes is essentially multivariate.
Let $K_2$ be the association scheme of size two
and $\mathfrak{X}=K_2 \otimes K_2 \otimes K_2$.
Then $\mathfrak{X}$ is a 7-class symmetric association scheme.
Since $\mathfrak{X}$ is isomorphic to
the group association scheme of $(\mathbb{F}_2^3, +)$,
the Bose-Mesner algebra of $\mathfrak{X}$ is isomorphic to
the group ring $\C [\mathbb{F}_2^3]$ as algebra.
Then for any two adjacency matrices $A_i ,A_j$ of $\mathfrak{X}$,
we have $\dim \Span_{\C}\{A_i^m A_j^n \mid m,n\in \N\} \le 4 < 8$.
By Lemma~\ref{lem:BMalg},
this implies that $\mathfrak{X}$ is not bivariate $P$-polynomial.
Therefore, $\mathfrak{X}$ is essentially trivariate $P$-polynomial.
Also, $\mathfrak{X}$ is essentially trivariate $Q$-polynomial.

\subsection{Composition of Gelfand pairs}
\label{sec:ComGelpair}
Let $G$ and $F$ be two finite groups with subgroups
$K \le G$ and $H \le F$.
Denote by $X = G/K$ and $Y = F/H$
the corresponding homogeneous spaces.
Recall that the \emph{wreath product} $F\wr G$ of $F$ by $G$ is the group whose set of elements is
\[
F^X\times G :=\{(f,g)\mid f\colon X\to F,\  g\in G\}    
\]
and multiplication $(f, g)(f ',g') = (f \cdot (gf'), gg')$, where
$[f\cdot (gf')](x) = f(x)f'(g^{-1}x)$
for all $x \in X$.
Consider the composition action of the wreath product $F\wr G$ on $X \times Y$ by
\[
(f,g)(x,y):=(gx,f(gx)y)
\]
for $(f,g)\in F\wr G$ and $(x,y)\in X\times Y$.
Let $x_0\in X$ and $y_0\in Y$ be the points stabilized by
$K$ and $H$, respectively.
By~\cite{CST2006},
the stabilizer $J\le F\wr G$ of the point $(x_0,y_0)$
is given by
\[
    J=\{(f,k)\in F\wr G \mid k\in K,\  f(x_0)\in H \}.
\]
Also let $X= \bigsqcup^n_{i=0}\Xi_i$
and $Y = \bigsqcup^m_{j=0} \Lambda_j$
be the decompositions of $X$ and $Y$ into their
$K$- (respectively $H$-) orbits
(with $\Xi_0 =\{x_0\}$ and $\Lambda_0 =\{y_0\}$).
By~\cite{CST2006}, the decomposition of $X\times Y$ into its $J$-orbits is given by
\[
X\times Y=
\left[\bigsqcup^m_{j=0} (\Xi_0\times \Lambda_j)\right]    
\sqcup
\left[\bigsqcup^n_{i=1} (\Xi_i\times Y)\right].
\]

Suppose that $(G, K)$ and $(F, H )$ are Gelfand pairs
and let $L(X) = \bigoplus ^n_{i=0} V_i$ and $L(Y ) = \bigoplus ^m_{j=0} W_j$
be the decomposition into $G$- (respectively $F$-)
irreducible subrepresentations,
where $V_0$ and $W_0$ are the one-dimensional subspaces of
constant functions.
By~\cite{CST2006},
$(F \wr G, J )$ is a Gelfand pair if and only if 
$(G, K)$ and $(F, H )$ are Gelfand pairs.
Moreover, the decomposition of $L(X \times Y)$ into
$(F \wr G)$-irreducibles is given by
\[
L(X\times Y)=  
\left[\bigoplus^n_{i=0} (V_i\otimes W_0)\right]    
\oplus
\left[\bigoplus^m_{j=1} (L(X)\otimes W_j)\right].
\]
Let $\mathfrak{X}$ and $\mathfrak{Y}$
be the association schemes obtained by $(G,K)$ and $(F,H)$, respectively.
The association scheme obtained by $(F \wr G, J )$
is called the \emph{composition} of $\mathfrak{X}$ and $\mathfrak{Y}$.
Note that the composition of $\mathfrak{X}$ and $\mathfrak{Y}$ is a fusion scheme of $\mathfrak{X}\otimes \mathfrak{Y}$.

\begin{theo}
    Let $\mathfrak{Z}$ be the composition of $\mathfrak{X}$ and $\mathfrak{Y}$. 
\begin{enumerate}[label=$(\roman*)$]
    \item If $\mathfrak{X}$ and $\mathfrak{Y}$ are $P$-polynomial association schemes,
    then $\mathfrak{Z}$ is a bivariate $P$-polynomial association scheme
    on
    \[
        \mathcal{D}:=\{(i,0)\}^n_{i=1}\cup \{(0,j)\}^m_{j=0}\subset \N^2
    \]
    with respect to $\le_{\mathrm{lex}}$.
    \item If $\mathfrak{X}$ and $\mathfrak{Y}$  are $Q$-polynomial association schemes,
    then $\mathfrak{Z}$ is a bivariate $Q$-polynomial association scheme on
    \[
        \mathcal{D}^\ast:=\{(j,0)\}^m_{j=1}\cup \{(0,i)\}^n_{i=0}\subset \N^2
    \]
    with respect to $\le_{\mathrm{lex}}$.
\end{enumerate}
\end{theo}

\begin{proof}
(i) 
Obviously, $\mathcal{D}$ satisfies (i) of Definition~\ref{df:abPpoly}.
Let $\{A_i\}^n_{i=0}$ and $\{A'_j\}^m_{j=0}$
be the adjacency matrices of $\mathfrak{X}$ and $\mathfrak{Y}$, respectively,
and $p^k_{ij}$ and $p'^k_{ij}$
be the intersection numbers of $\mathfrak{X}$ and $\mathfrak{Y}$, respectively.
Let $R_{i0}:=\widetilde{\Xi_i \times Y}$ and $R_{0j}:=\widetilde{\Xi_0 \times \Lambda_j}$.
By definition, the adjacency matrices of $\mathfrak{Z}$ are
$A_{i0}=A_i\otimes J_Y$
for $(i,0)\in \mathcal{D}$ and
$A_{0j}=I_X\otimes A'_j$
for $(0,j)\in \mathcal{D}$.
Since $\mathfrak{X}$ and $\mathfrak{Y}$ are $P$-polynomial, we have
\begin{align*}
    A_{01} A_{0j} &= I_X^2 \otimes A'_1 A'_j = I_X\otimes (p'^{j-1}_{1j} A'_{j-1} + p'^{j}_{1j} A'_j +p'^{j+1}_{1j} A'_{j+1})\\
    &=p'^{j-1}_{1j} A_{0,j-1} + p'^{j}_{1j} A_{0j}+p'^{j+1}_{1j} A_{0,j+1}\ \text{for}\  0\le j \le m,\\
    A_{10} A_{0j} &= A_1 \otimes J_Y A'_j = p'^{0}_{jj} A_1 \otimes J_Y =p'^{0}_{jj} A_{10}\ \text{for}\  0\le j \le m,\\
    A_{01} A_{i0} &= A_i \otimes A'_1 J_Y = p'^{0}_{11} A_i \otimes J_Y = p'^{0}_{11} A_{i0}\ \text{for}\  1\le i \le n,\\
    A_{10} A_{10} &= A_1^2 \otimes J_Y^2 = |Y| (p^0_{11} I_X + p^1_{11} A_1 +p^2_{11} A_2)\otimes J_Y\\
    &=|Y| (p^0_{11} \sum^m_{j=0} A_{0j} + p^1_{11} A_{10} +p^2_{11} A_{20}),\\
    A_{10} A_{i0} &= A_1 A_i \otimes J_Y^2 = |Y| (p^{i-1}_{1i} A_{i-1} + p^{i}_{1i} A_i +p^{i+1}_{1i} A_{i+1})\otimes J_Y\\
    &=|Y| (p^{i-1}_{1i} A_{i-1,0} + p^{i}_{1i} A_{i0} +p^{i+1}_{1i} A_{i+1,0})\ \text{for}\  i\ge 2.
\end{align*}
By Proposition~\ref{prop:P-TFAE}, $\mathfrak{Z}$ is
a bivariate $P$-polynomial association scheme on $\mathcal{D}$
with respect to $\le_{\mathrm{lex}}$.

(ii)
Obviously, $\mathcal{D}^\ast$ satisfies (i) of Definition~\ref{df:abQpoly}.
Let $\{E_i\}^n_{i=0}$ and $\{E'_j\}^m_{j=0}$
be the primitive idempotents of $\mathfrak{X}$ and $\mathfrak{Y}$, respectively,
and $q^k_{ij}$ and $q'^k_{ij}$
be the Krein numbers of $\mathfrak{X}$ and $\mathfrak{Y}$, respectively.
Let $E_{j0}$ and $E_{0i}$
be the primitive idempotents with respect to $L(X)\otimes W_j$
and $V_i\otimes W_0$, respectively.
Then we have
$E_{j0}=I_X\otimes E'_j$
for $(j,0)\in \mathcal{D}^\ast$
and
$E_{0i}=E_i\otimes E'_0$
for $(0,i)\in \mathcal{D}^\ast$.
Also, we put $F_{j0}=|X||Y|E_{j0}$ and $F_{0i}=|X||Y|E_{0i}$.
Since $\mathfrak{X}$ and $\mathfrak{Y}$ are $Q$-polynomial, we have
\begin{align*}
    F_{01} \circ F_{0i} &= (|X|E_1 \circ |X|E_i) \otimes (|Y|E'_{0} \circ |Y|E'_{0}) 
    =|X|(q^{i-1}_{1i}E_{i-1}+q^i_{1i}E_i +q^{i+1}_{1i} E_{i+1})\otimes |Y|E'_{0}\\
    &=q^{i-1}_{1i}F_{0,i-1}+q^i_{1i}F_{0i} +q^{i+1}_{1i} F_{0,i+1}\ \text{for}\  0\le i \le n,\\
    F_{10} \circ F_{0i} &= (|X|I_X \circ |X|E_i) \otimes (|Y|E'_{1} \circ |Y|E'_{0}) \\
    &=(q^0_{ii}|X|I_X)\otimes (|Y|E'_1) = q^0_{ii} F_{10}\ \text{for}\  0\le i \le n,\\
    F_{01}\circ F_{j0} &= (|X|E_1 \circ |X|I_X) \otimes (|Y|E'_{0} \circ |Y|E'_{j}) \\
    &=(q^0_{11}|X|I_X)\otimes (|Y|E'_j) = q^0_{11} F_{j0}\ \text{for}\  0\le j \le m,\\
    F_{10}\circ F_{10} &= (|X|I_X\circ |X|I_X)\otimes (|Y|E'_1 \circ |Y|E'_1) =|X|^2 I_X\otimes |Y|(q'^0_{11}E'_0 + q'^1_{11}E'_1 + q'^2_{11}E'_2)\\
    &=|X|(q'^0_{11} \sum^n_{i=0}F_{0i} + q'^1_{11}F_{10} + q'^2_{11}F_{20}),\\
    F_{10}\circ F_{j0} &= (|X|I_X\circ |X|I_X) \otimes (|Y|E'_1 \circ |Y|E'_j)= |X|^2 I_X\otimes|Y| (q'^{j-1}_{1j}E'_{j-1}+q'^j_{1j}E'_j +q'^{j+1}_{1j}E'_{j+1})\\
    &=|X|(q'^{j-1}_{1j}F_{j-1,0}+q'^j_{1j}F_{j0} +q'^{j+1}_{1j}F_{j+1,0})\ \text{for}\  j\ge 2.
\end{align*}
By Proposition~\ref{prop:Q-TFAE}, $\mathfrak{Z}$ is
a bivariate $Q$-polynomial association scheme on $\mathcal{D}^\ast$
with respect to $\le_{\mathrm{lex}}$.
\end{proof}

\subsection{The extensions of association schemes}
\label{sec:GHS}

In this subsection, we show that the extensions of association schemes are multivariate $P$-polynomial and $Q$-polynomial 
association schemes.
For more information on extensions of association schemes,
see \cite{bi,CST2006,Delsarte1973,MT2004}.

For a commutative association scheme $\mathfrak{X}=(X,\{R_i\}^\ell_{i=0})$ of class $\ell$
and an integer $n\ge 1$, let $X^n$ be the $n$-th Cartesian power of $X$.
For $x=(x_1,x_2,\ldots , x_n),y=(y_1,y_2,\ldots , y_n)\in X^n$
and $0\le i \le \ell$,
set $\tau_i(x,y):=|\{t =1,2,\ldots ,n\mid (x_t,y_t)\in R_i\}|$.
We define the following $\ell$-tuple:
\[
    \mathcal{R} (x,y) :=(\tau_1(x,y),\tau_2(x,y),\ldots ,\tau_\ell(x,y)).    
\]
By the definition of $\mathcal{R}$,
all $\mathcal{R} (x,y)$ are in $\mathcal{D}:=\{\alpha \in \N^{\ell} \mid |\alpha|\le n \}$. 
Then $\mathcal{S}^n(\mathfrak{X}):=(X^n,\mathcal{R})$ is a commutative association scheme
and called the \emph{extension}
(cf.~Delsarte~\cite{Delsarte1973})
or the \emph{symmetrization}
\footnote{In general, $\mathcal{S}^n(\mathfrak{X})$ is not symmetric (``symmetric'' means the condition~\ref{AS:symmetric}).
The origin of ``symmetrization'' seems to come from the fact that the Bose-Mesner algebra of $\mathcal{S}^n(\mathfrak{X})$ is the symmetric tensor subspace of the Bose-Mesner algebra of $\bigotimes^n \mathfrak{X}$.}
(cf.~Bernard et al.~\cite{bi})
of $\mathfrak{X}$ of length $n$.
Note that $\mathcal{S}^n(\mathfrak{X})$ is a fusion scheme of the $n$-times direct product $\bigotimes^n \mathfrak{X}$
of $\mathfrak{X}$.

\begin{Rem}
If $\mathfrak{X}$ comes from a finite Gelfand pair $(F, H)$,
then $\mathcal{S}^n(\mathfrak{X})$ is equivalent to a Gelfand pair $(F \wr \mathfrak{S}_n, H \wr \mathfrak{S}_n) = (F^n \ltimes \mathfrak{S}_n, H^n \ltimes \mathfrak{S}_n)$,
where $\mathfrak{S}_n$ is the symmetric group on $\{1,2,\ldots , n\}$.
In~\cite{CST2006}, $\mathcal{S}^n(\mathfrak{X})$ is called a \emph{generalized Hamming scheme}.    
In fact, when $F=\mathfrak{S}_{q}$ and $H=\mathfrak{S}_{q-1}$,
i.e, $\mathfrak{X}$ is the complete graph $K_q$ of size $q$,
$\mathcal{S}^n(\mathfrak{X})$ is the Hamming scheme $H(n,q)$.
\end{Rem}

Let $\{A_i\}^\ell_{i=0}$ and $\{E_j\}^\ell_{j=0}$ be the adjacency matrices and the primitive idempotents of $\mathfrak{X}$, respectively.
For $\alpha=(n_1,n_2,\ldots ,n_\ell)\in \mathcal{D}$,
the adjacency matrix $A_\alpha$ of $\mathcal{S}^n(\mathfrak{X})$ is
\[
A_\alpha =  \frac{1}{(n-|\alpha|)! \prod^\ell_{i=1}n_i!}\sum_{\pi\in \mathfrak{S}_n} \pi \cdot A_1^{\otimes n_1}\otimes A_2^{\otimes n_2}\otimes \cdots \otimes A_\ell^{\otimes n_\ell} \otimes A_0^{\otimes (n-|\alpha|)},
\]
where the sum is over all the place permutations,
and the prefactor ensures that each term appears only once.
Also, the primitive idempotents of $\mathcal{S}^n(\mathfrak{X})$
are indexed by $\mathcal{D}$, and we have
\[
E_\alpha =  \frac{1}{(n-|\alpha|)! \prod^\ell_{i=1}n_i!}\sum_{\pi\in \mathfrak{S}_n} \pi \cdot E_1^{\otimes n_1}\otimes E_2^{\otimes n_2}\otimes \cdots \otimes E_\ell^{\otimes n_\ell} \otimes E_0^{\otimes (n-|\alpha|)}.
\]

\begin{theo}
    \label{thm:GHS}
    $\mathcal{S}^n(\mathfrak{X})$ is an $\ell$-variate $P$-polynomial and $Q$-polynomial association scheme
    on $\mathcal{D}$ with respect to $\le_{\mathrm{grlex}}$.
\end{theo}

\begin{proof}
Firstly we will show that $\mathcal{S}^n(\mathfrak{X})$ is an $\ell$-variate $P$-polynomial.
Obviously, $\mathcal{D}$ satisfies (i) of Definition~\ref{df:abPpoly}.
Fix $i\in \{1,2,\ldots ,\ell\}$ and $\alpha\in \mathcal{D}$.
We calculate $A_{\epsilon_i} A_\alpha$ as follows.
By the equations
\[
    A_{\epsilon_i} =
    \sum^n_{j=1} A_0^{\otimes (j-1)}\otimes A_i \otimes A_0^{\otimes (n-j)}
\]
and
\begin{align}
    &(A_0^{\otimes (j-1)}\otimes A_i \otimes A_0^{\otimes (n-j)})
    (A_{i_1}\otimes A_{i_2}\otimes \cdots  \otimes A_{i_n})\notag\\
    &=A_{i_1}\otimes A_{i_2}\otimes \cdots \otimes A_{i_{j-1}} \otimes (A_i A_{i_j}) \otimes A_{i_{j+1}}\otimes \cdots \otimes A_{i_n}\notag\\
    &= \sum^\ell_{m=0} p^{m}_{i,i_j} A_{i_1}\otimes A_{i_2}\otimes \cdots \otimes A_{i_{j-1}} \otimes A_m \otimes A_{i_{j+1}}\otimes \cdots \otimes A_{i_n}
    \label{eq:otimes_prod}
\end{align}
for $i_1,i_2,\ldots ,i_n \in \{0,1,\ldots ,\ell\}$,
the possible terms that appear in the expansion of $A_{\epsilon_i} A_\alpha$ with respect to $\{A_\alpha\}_{\alpha\in \mathcal{D}}$
are either
\[
A_{\alpha+\epsilon_i},\  
A_{\alpha},\  
A_{\alpha-\epsilon_i+\epsilon_s},\  
A_{\alpha+\epsilon_i-\epsilon_s},\  
A_{\alpha-\epsilon_s+\epsilon_t}\  
\text{or}\  
A_{\alpha-\epsilon_i},
\]
where $s, t \in \{1,2,\ldots , \ell\}$ and $i, s, t$ are all distinct.
Here, we determine the coefficient of $A_{\alpha+\epsilon_i}$ in $A_{\epsilon_i} A_\alpha$.
Fix the term 
\begin{equation}
    \label{eq:a+ei}
    A_0^{\otimes (n-|\alpha|-1)} \otimes A_1^{\otimes n_1}\otimes \cdots \otimes A_i^{\otimes (n_i+1)} \otimes \cdots \otimes A_\ell^{\otimes n_\ell}
\end{equation}
of $A_{\alpha+\epsilon_i}$.
Then the indices in \eqref{eq:otimes_prod} of shape \eqref{eq:a+ei} are
$i_j=0$ and $m=i$.
This leads to $p^i_{i0}=1$.
There are $n_i+1$ choices of positions where $A_0$ changes to $A_i$.
Thus, the coefficient of $A_{\alpha+\epsilon_i}$ in $A_{\epsilon_i} A_\alpha$
is $n_i+1$.
Since the other coefficients can be obtained by similar calculations, the following equation is obtained:
\begin{align}
    A_{\epsilon_i} A_\alpha
    =&
    (n_i+1)A_{\alpha+\epsilon_i}
    +\sum^\ell_{j=1}n_j p^j_{ij}A_\alpha+\sum_{s\neq i}(n_s+1)p^s_{ii}A_{\alpha -\epsilon_i +\epsilon_s} \notag\\
    &+\sum_{s\neq i}(n_i+1)p^i_{is}A_{\alpha +\epsilon_i -\epsilon_s}
    +\sum_{s,t\neq i}(n_t+1)p^t_{is}A_{\alpha -\epsilon_s +\epsilon_t}
    +(n-|\alpha|+1)p^0_{ii}A_{\alpha-\epsilon_i}. \label{eq:GHS}
\end{align}
Note that 
in the right-hand side of \eqref{eq:GHS},
the terms whose indices do not belong to $\mathcal{D}$ are
regarded as not appearing.
Thus, these indices appearing on the right-hand side of \eqref{eq:GHS}
are less than or equal to $\alpha+\epsilon_i$
with respect to $\le_\mathrm{grlex}$.
Moreover, if $\alpha+\epsilon_i\in \mathcal{D}$, 
then $p^{\alpha+\epsilon_i}_{\epsilon_i,\alpha}=n_i+1 \neq 0$.
By Proposition~\ref{prop:P-TFAE}, $\mathcal{S}^n(\mathfrak{X})$ is
an $\ell$-variate $P$-polynomial association scheme on $\mathcal{D}$
with respect to $\le_{\mathrm{grlex}}$.

For a proof that $\mathcal{S}^n(\mathfrak{X})$ is $\ell$-variate $Q$-polynomial,
it can be shown in the same way as the above proof of the $\ell$-variate $P$-polynomial property of $\mathcal{S}^n(\mathfrak{X})$.
\end{proof}

\subsection{The generalized Johnson schemes}
\label{sec:GJS}

In this subsection, we show that the generalized Johnson schemes are multivariate $P$-polynomial association schemes.
For more information on generalized Johnson schemes,
see \cite{CST2006}.

Let $(F,H)$ be a finite Gelfand pair,
$Y=F/H$ and $y_0\in Y$ the point stabilized by $H$.
Suppose that $Y = \bigsqcup^m_{i=0} \Lambda_i$ is the decomposition of $Y$
into its $H$-orbits with $\Lambda_0=\{y_0\}$.
For $0\le h \le n$,
let $\Omega_h$ be the $\mathfrak{S}_n$-homogeneous space
($\cong \mathfrak{S}_n/\mathfrak{S}_{n-h}\times \mathfrak{S}_h$)
consisting of all $h$-subsets of $\{1, 2,\ldots, n\}$.
We consider the wreath product $F \wr \mathfrak{S}_n$ of $F$ and $\mathfrak{S}_n$
(with respect to the action of $\mathfrak{S}_n$ on $\{1, 2,\ldots, n\}$)
and we construct a natural homogeneous space of $F \wr \mathfrak{S}_n$
using the actions of $F$ on $Y$ and of $\mathfrak{S}_n$ on $\Omega_h$.

Let $\Theta_h$ be the set of all functions $\theta \colon A \to Y$
whose domain is an element of $\Omega_h$
($A\in \Omega_h$)
and whose range is $Y$.
In other words $\Theta_h=\bigsqcup_{A\in \Omega_h} Y^A$.
If $\theta\in \Theta_h$ and $\theta \colon A \to Y$
then we will write $\dom \theta = A$.
The group $F \wr \mathfrak{S}_n$ acts on $\Theta_h$ in a natural way:
if $(f,\pi)\in F \wr \mathfrak{S}_n$ and $\theta\in \Theta_h$ then $(f,\pi)\theta$ is the function,
with domain $\pi \dom \theta$,
defined by setting
\[
    \left[(f,\pi)\theta\right](j)=f(j)\theta(\pi^{-1} j)
\]
for every $j\in \pi \dom \theta$.
It is clear that this action is transitive.

If $\bar{A}$ is the element in $\Omega_h$
stabilized by $\mathfrak{S}_{n-h}\times \mathfrak{S}_h$, and
we define $\theta_0\in Y^{\bar{A}}\subseteq \Theta_h$
by setting $\theta_0(j) = y_0$ for every $j \in \bar{A}$,
then is easy to check that the stabilizer of $\theta_0$ is equal to
$(H \wr \mathfrak{S}_h) \times (F \wr \mathfrak{S}_{n-h})$;
therefore we can write
$\Theta_h=(F \wr \mathfrak{S}_n)/[(H \wr \mathfrak{S}_h) \times (F \wr \mathfrak{S}_{n-h})]$.
An element $\tilde{\alpha}=(a_0,a_1,\ldots,a_m)\in \N^{m+1}$
is called a \emph{weak $(m + 1)$-composition} of $h$ if $\tilde{\alpha}$ satisfies $|\tilde{\alpha}|=h$.
In what follows, the set of all weak $(m + 1)$-compositions of $h$ will be denoted by $C(h,m + 1)$,
and it is obviously that $|C(h,m+ 1)|=\binom{m+h}{m}$.

\begin{dfe}[Change $t \to h-t$ of Definition~6.3 in \cite{CST2006}]
    For $\theta\in \Theta_h$ we define the \emph{type} of $\theta$ as the sequence of nonnegative integers
    $\type(\theta ) = (t, a_0,a_1,\ldots,a_m)$
     where $t =h-|\dom \theta \cap \bar{A}|$ and
     $a_i =|\{j \in \dom \theta \cap \bar{A} \mid  \theta(j) \in \Lambda_i\}|$,
     $i = 0, 1,\ldots,m$.
\end{dfe}

By~\cite{CST2006}, $(F \wr \mathfrak{S}_{n},(H \wr \mathfrak{S}_h) \times (F \wr \mathfrak{S}_{n-h}))$
is a Gelfand pair, and it is shown that
the orbits of $(H \wr \mathfrak{S}_h) \times (F \wr \mathfrak{S}_{n-h})$ on $\Theta_h$
are parameterized by the set
\begin{equation}
\label{eq:t_alpha}    
\{(t,\tilde{\alpha})\mid  0\le t \le \min\{h,n-h\},\   \tilde{\alpha}\in C(h-t,m+1)\}.
\end{equation}

A map $\mathcal{R}$ from $\Theta_h\times \Theta_h$
to \eqref{eq:t_alpha} is defined by
$\mathcal{R}(\theta_1,\theta_2)=(t,\tilde{\alpha})$
if
$|\dom \theta_1 \cap \dom \theta_2|=h-t$ and
\[
    |\{a\in\dom \theta_1 \cap \dom \theta_2 \mid (\theta_1(a),\theta_2(a))\in \widetilde{\Lambda_i}\}|=a_i,
\]
for $i=0,1,\ldots ,m$.
Thus, $\mathfrak{X}=(\Theta_h,\mathcal{R})$ is a commutative association scheme.
Let
\[
    \mathcal{D}:=\{(t,a_1,a_2, \ldots , a_m)\in \N^{m+1} \mid 0\le t \le \min\{h,n-h\},\ a_i \ge 0,\  \sum^m_{i=1}a_i \le h-t\}.
\]
For $\tilde{\alpha}=(a_0,a_1, \ldots , a_m)$,
we put $\alpha=(a_1,a_2, \ldots , a_m)$.
One can check easily that there is a one-to-one correspondence between \eqref{eq:t_alpha} and $\mathcal{D}$.
Hence, we also treat $\mathcal{R}$ as a map from $\Theta_h\times \Theta_h$ to $\mathcal{D}$
by $\mathcal{R}(\theta_1,\theta_2)=(t,\alpha)$.

\begin{Rem}
When the case of $F=H$, i.e., $Y$ is a singleton,
then $\mathfrak{X}$ coincides with the Johnson scheme $J(n,h)$.
On the other hand,
when the case of $F=\mathfrak{S}_{q}$
and $H=\mathfrak{S}_{q-1}$, i.e., $Y$ is the complete graph $K_{q}$, then
$\mathfrak{X}$ coincides with the nonbinary Johnson scheme
$J_q(n,h)$.
For more information on nonbinary Johnson schemes,
see Tarnanen-Aaltonen-Goethals~\cite{TAG1985}.
\end{Rem}

\begin{theo}
\label{thm:GJS}
$\mathfrak{X}=(\Theta_h,\mathcal{R})$
is an $(m+1)$-variate $P$-polynomial association scheme
on $\mathcal{D}$
with respect to $\le_\mathrm{grlex}$.
\end{theo}

\begin{proof}
One can easily check that $\mathcal{D}$
satisfies (i) of Definition~\ref{df:abPpoly}.

Fix $(t,\alpha)=(t,a_1,a_2,\ldots ,a_m)\in \mathcal{D}$.
Let $A_{(t,\alpha)}$ be the adjacency matrix
of $R_{(t,\alpha)}$.
Firstly, we determine non-zero coefficients $p^{(u,\beta)}_{(1,\zerovec)(t,\alpha)}$ of 
\[A_{(1,\zerovec)} A_{(t,\alpha)}=\sum_{(u,\beta)\in \mathcal{D}} p^{(u,\beta)}_{(1,\zerovec)(t,\alpha)} A_{(u,\beta)}.
\]
Assume $p^{(u,\beta)}_{(1,\zerovec)(t,\alpha)} \neq 0$.
Take $\theta_1,\theta_2,\theta_3\in \Theta_h$
satisfying $(\theta_1,\theta_2)\in R_{(u,\beta)}$,
$(\theta_1,\theta_3)\in R_{(1,\zerovec)}$
and $(\theta_3,\theta_2)\in R_{(t,\alpha)}$.
Then $|\dom \theta_1 \cap \dom \theta_2|=h-u$,
$|\dom \theta_1 \cap \dom \theta_3|=h-1$
and $|\dom \theta_3 \cap \dom \theta_2|=h-t$
hold.
Then $\dom \theta_1,\dom \theta_2,\dom \theta_3$
are regarded as elements of the Johnson scheme $J(n,h)$.
Since $J(n,h)$
is $P$-polynomial,
$|t-u|\le 1$ holds,
i.e., $u=t-1,t,t+1$ holds.

(1a) The case $u=t-1$.
By the relation of $\dom \theta_1,\dom \theta_2,\dom \theta_3$,
the domain of $\theta_3$ must be 
\[
    \dom \theta_3 = (\dom \theta_1\setminus\{a\}) \cup \{b\}, 
\]
where $a\in \dom \theta_1 \cap \dom \theta_2$
and $b\in \{1,2,\ldots ,n\}\setminus (\dom \theta_1 \cup \dom \theta_2)$.
Since $(\theta_1,\theta_3)\in R_{(1,\zerovec)}$,
we have $\theta_1(c)=\theta_3(c)$
for each $c\in \dom \theta_1\setminus\{a\}$.
Thus,
$(\theta_3(c),\theta_2(c))$ and $(\theta_1(c),\theta_2(c))$
are in the same relation on $Y$ 
for each $c\in (\dom \theta_1\cap \dom \theta_2)\setminus\{a\}
=\dom \theta_3 \cap \dom \theta_2$.

If $\theta_1(a)=\theta_2(a)$, i.e., $(\theta_1(a),\theta_2(a))\in \widetilde{\Lambda_0}$ on $Y$,
then 
\begin{align*}
    &|\{c\in \dom \theta_1 \cap \dom \theta_2 \mid (\theta_1(c),\theta_2(c))\in \widetilde{\Lambda_0}\}|=
    \left((h-t)-|\alpha|\right)+1\ 
    \text{and}\\
    &|\{c\in \dom \theta_1 \cap \dom \theta_2 \mid (\theta_1(c),\theta_2(c))\in \widetilde{\Lambda_i}\}|=
    a_i\  \text{for $i=1,2,\ldots ,m$.}
\end{align*}
This implies $\mathcal{R}(\theta_1,\theta_2)=(t-1,\alpha)$.

If $(\theta_1(a),\theta_2(a))\in \widetilde{\Lambda_i}$ on $Y$,
for each $i=1,2,\ldots ,m$,
then 
\begin{align*}
    &|\{c\in \dom \theta_1 \cap \dom \theta_2 \mid (\theta_1(c),\theta_2(c))\in \widetilde{\Lambda_0}\}|=
    (h-t)-|\alpha|,\\ 
    &|\{c\in \dom \theta_1 \cap \dom \theta_2 \mid (\theta_1(c),\theta_2(c))\in \widetilde{\Lambda_i}\}|=
    a_i+1\ 
    \text{and}\\ 
    &|\{c\in \dom \theta_1 \cap \dom \theta_2 \mid (\theta_1(c),\theta_2(c))\in \widetilde{\Lambda_j}\}|=
    a_j\  \text{for $j\in \{1,2,\ldots ,m\}\setminus \{i\}$.}
\end{align*}
This implies 
$\mathcal{R}(\theta_1,\theta_2)=(t-1,\alpha+\epsilon_i)$.
Hence, if $p^{(t-1,\beta)}_{(1,\zerovec),(t,\alpha)}\neq 0$, then $\beta=\alpha$ or $\alpha+\epsilon_i$.
One can check the exact values of the intersection numbers as follows:
\[
p^{(t-1,\alpha)}_{(1,\zerovec),(t,\alpha)}
=|Y|(n-h-t+1)(h-t+1-|\alpha|)
\ \text{and}\  
p^{(t-1,\alpha+\epsilon_i)}_{(1,\zerovec),(t,\alpha)}
=|Y|(n-h-t+1)(a_i+1).
\]

(1b) The case $u=t$.
By the relation of $\dom \theta_1,\dom \theta_2,\dom \theta_3$,
the domain of $\theta_3$ must be 
\[
    \dom \theta_3 = (\dom \theta_1\setminus\{a\}) \cup \{b\}, 
\]
where
\begin{equation}
    \label{eq:u=t1}
    \text{$a\in \dom \theta_1 \setminus \dom \theta_2$
    and $b\in \{1,2,\ldots ,n\}\setminus (\dom \theta_1 \cup \dom \theta_2)$}        
\end{equation}
or
\begin{equation}
    \label{eq:u=t2}
    \text{$a\in \dom \theta_1 \cap \dom \theta_2$
    and $b\in \dom \theta_2 \setminus \dom \theta_1$}.        
\end{equation}
With a similar argument as the case $u=t-1$,
we get that
$\beta=\alpha$ for \eqref{eq:u=t1}
and
$\beta=\alpha,
\alpha \pm \epsilon_i,
\alpha+ \epsilon_i-\epsilon_j,
$ ($i,j=1,2,\ldots ,m$)
for \eqref{eq:u=t2}.
One can check the exact values of the intersection numbers as follows:
\[
p^{(t,\alpha)}_{(1,\zerovec),(t,\alpha)}
=|Y|(n-h-t)t+(h-t-|\alpha|)t+\sum^m_{j=1}a_j t k_j,\  
p^{(t,\alpha-\epsilon_i)}_{(1,\zerovec),(t,\alpha)}
=(h-t+1-|\alpha|)t k_i,
\]
\[
p^{(t,\alpha+\epsilon_i)}_{(1,\zerovec),(t,\alpha)}
=(a_i+1) t\ 
\text{and}\  
p^{(t,\alpha+\epsilon_i-\epsilon_j)}_{(1,\zerovec),(t,\alpha)}
=(a_i+1) t k_j,        
\]
where $k_i$ is the valency of $\widetilde{\Lambda_i}$ of $Y$.

(1c) The case $u=t+1$.
By the relation of $\dom \theta_1,\dom \theta_2,\dom \theta_3$,
the domain of $\theta_3$ must be 
\[
    \dom \theta_3 = (\dom \theta_1\setminus\{a\}) \cup \{b\}, 
\]
where
$a\in \dom \theta_1 \setminus \dom \theta_2$
and $b\in \dom \theta_2 \setminus \dom \theta_1$.
With a similar argument as the case $u=t-1$,
we get that
$\beta=\alpha, \alpha-\epsilon_i$.
One can check the exact values of the intersection numbers as follows:
\[
p^{(t+1,\alpha)}_{(1,\zerovec),(t,\alpha)}
=(t+1)^2 \ \text{and}\  
p^{(t+1,\alpha-\epsilon_i)}_{(1,\zerovec),(t,\alpha)}
=(t+1)^2 k_i.
\]
By (1a)--(1c),
we have
\begin{align}
    A_{(1,\zerovec)} A_{(t,\alpha)}
    =&
    p^{(t+1,\alpha)}_{(1,\zerovec),(t,\alpha)}A_{(t+1,\alpha)}+
    \sum^m_{i=1}p^{(t+1,\alpha-\epsilon_i)}_{(1,\zerovec),(t,\alpha)}A_{(t+1,\alpha-\epsilon_i)}\notag\\
    &+p^{(t,\alpha)}_{(1,\zerovec),(t,\alpha)}A_{(t,\alpha)}+
    \sum^m_{i=1}p^{(t,\alpha-\epsilon_i)}_{(1,\zerovec),(t,\alpha)}A_{(t,\alpha-\epsilon_i)}\notag\\
    &+\sum^m_{i=1}p^{(t,\alpha+\epsilon_i)}_{(1,\zerovec),(t,\alpha)}A_{(t,\alpha+\epsilon_i)}+
    \sum_{1\le i,j\le m, i\neq j}p^{(t,\alpha-\epsilon_i+\epsilon_j)}_{(1,\zerovec),(t,\alpha)}A_{(t,\alpha-\epsilon_i+\epsilon_j)}\notag\\
    &+p^{(t-1,\alpha)}_{(1,\zerovec),(t,\alpha)}A_{(t-1,\alpha)}+\sum^m_{i=1}p^{(t-1,\alpha+\epsilon_i)}_{(1,\zerovec),(t,\alpha)}A_{(t-1,\alpha+\epsilon_i)}.
    \label{eq:(10)(tb)}
\end{align}

Fix $i=1,2,\ldots, m$.
Next, we determine non-zero coefficients $p^{(u,\beta)}_{(0,\epsilon_i)(t,\alpha)}$ of 
\[A_{(0,\epsilon_i)} A_{(t,\alpha)}=\sum_{(u,\beta)\in \mathcal{D}} p^{(u,\beta)}_{(0,\epsilon_i)(t,\alpha)} A_{(u,\beta)}.
\]
With a similar argument as $A_{(1,\zerovec)} A_{(t,\alpha)}$,
we get that
\begin{align}
    A_{(0,\epsilon_i)} A_{(t,\alpha)}
    =&
    p^{(t,\alpha+\epsilon_i)}_{(0,\epsilon_i),(t,\alpha)}A_{(t,\alpha+\epsilon_i)}
    +
    p^{(t,\alpha)}_{(0,\epsilon_i),(t,\alpha)}A_{(t,\alpha)} \notag\\
    &+
    \sum_{1\le j,k\le m, j\neq k}
    p^{(t,\alpha+\epsilon_j-\epsilon_k)}_{(0,\epsilon_i),(t,\alpha)}A_{(t,\alpha+\epsilon_j-\epsilon_k)}
    +p^{(t,\alpha-\epsilon_i)}_{(0,\epsilon_i),(t,\alpha)}A_{(t,\alpha-\epsilon_i)},
    \label{eq:(01)(tb)}
\end{align}
where
\begin{align*}
    p^{(t,\alpha+\epsilon_i)}_{(0,\epsilon_i),(t,\alpha)}&=(a_i+1),\ 
    p^{(t,\alpha)}_{(0,\epsilon_i),(t,\alpha)}=
    (|Y|-1)t+\sum^m_{j=1}a_j(k_j-1),\\
    p^{(t,\alpha+\epsilon_j-\epsilon_k)}_{(0,\epsilon_i),(t,\alpha)}&=
    (a_k+1)k_j\   \text{and}\    
    p^{(t,\alpha-\epsilon_i)}_{(0,\epsilon_i),(t,\alpha)}=
    (h-t+1-|\alpha|)k_i.
\end{align*}
Note that 
in the right-hand sides of \eqref{eq:(10)(tb)} and \eqref{eq:(01)(tb)},
the terms whose indices do not belong to $\mathcal{D}$ are
regarded as not appearing.
Thus, these indices appearing in the right-hand side of \eqref{eq:(10)(tb)} and \eqref{eq:(01)(tb)} 
are less than or equal to $(1,\zerovec)+(t,\alpha)=(t+1,\alpha)$ and $(0,\epsilon_i)+(t,\alpha)=(t,\alpha+\epsilon_i)$
with respect to $\le_\mathrm{grlex}$, respectively.
Moreover, if $(t+1,\alpha)\in \mathcal{D}$, 
then $p^{(t+1,\alpha)}_{(1,\zerovec), (t,\alpha)}=(t+1)^2 \neq 0$
and if $(t,\alpha+\epsilon_i)\in \mathcal{D}$, 
then $p^{(t,\alpha+\epsilon_i)}_{(0,\epsilon_i), (t,\alpha)}= a_i +1 \neq 0$.
Therefore, by Proposition~\ref{prop:P-TFAE}, $\mathfrak{X}$ is
an $(m+1)$-variate $P$-polynomial association scheme on $\mathcal{D}$
with respect to $\le_\mathrm{grlex}$.
\end{proof}

\begin{Rem}
    The formula of spherical functions of the generalized Johnson schemes
    are given in \cite{CST2006}.
    From the formula,
    it can be expected that the generalized Johnson schemes are $(m+1)$-variable $Q$-polynomial association schemes.
    However, at present there is no proof of this.
\end{Rem}

\section{Further comments}
\label{sec:Furthercomments}

(i) In this paper, we did not discuss the multivariate $Q$-polynomial 
property for many of the association schemes that proved to be bivariate (or multivariate) $P$-polynomial association schemes. 
In fact, we believe that most of them satisfy this property. 
On the other hand, we have not been completely successful in proving this. So, we would like to leave this question to a subsequent paper.

What we want to prove is the following. For most of the examples of multivariate $P$-polynomial association schemes, such as, 
nonbinary Johnson schemes, association schemes obtained from attenuated space, association schemes obtained from $m$-dimensional isotropic subspaces, generalized 
Johnson schemes, etc., what we need to show is that there exist appropriate polynomials $v_{rs}^\ast(x,y)$ such that 
the following assertions hold. Let $\theta_{rs}^\ast$ be 
the dual eigenvalues of $|X|E_{rs}.$ Then $|X|E_{rs}=
v_{rs}^\ast(|X|E_{10},|X|E_{01})$ (where the multiplication 
is the circle product), namely 
$\theta_{rs}^\ast=v_{rs}^\ast(\theta_{10}^\ast,\theta_{01}^\ast).$ 
Since all the values of $\theta_{rs}^\ast$ are known very explicitly, 
say see formula (4-2) in Theorem 4.2 in Dunkl~\cite{Dunkl1976}
or formula (38) in Theorem 2 in \cite{TAG1985}
for the nonbinary Johnson association scheme, 
it should be possible in principle to obtain the claim, although currently we 
have difficulty in completing this proof rigorously. 

\bigskip

\noindent
(ii) Iliev-Terwilliger~\cite{IT2012} consider some multivariate $P$-polynomial 
(and/or $Q$-polynomial) association schemes from the viewpoint 
of root systems, in particular of type $A_n$ and possibly for 
other types. These are very special classes of more general multivariate 
$P$-polynomial (and/or $Q$-polynomial) association schemes 
we have considered. We expect many of our examples can be 
regarded as falling into this special class, and we hope to discuss 
more from this viewpoint in a subsequent paper. 

\bigskip

\noindent
(iii) There are known many multivariate orthogonal 
polynomials that are generalizations of $q$-Racah (Askey-Wilson) 
polynomials, as well as Hahn or dual Hahn polynomials, etc.  
For example, Scarabotti~\cite{Scarabotti2011} (as well as many authors) considers 
such a generalization. 
It seems that those considered by Gasper-Rahman~\cite{GR2007}, for instance, are nothing but special cases considered 
in Scarabotti~\cite{Scarabotti2011}. 
It would be interesting to see which of such multivariate orthogonal polynomials actually have association schemes in the underlying structure, or weaker linear algebraic structures such as character algebras, table algebras, hypergroups, etc.

\vspace{2ex}

\noindent
{\bf{Acknowledgments}}

\vspace{2ex}

\noindent
The first author thanks Luc Vinet for sending him their 
paper \cite{bi} before publication. Our present paper 
was deeply motivated by the impact of \cite{bi}. 
Also, we sincerely thank Nicolas Cramp\'{e}, Pierre-Antoine Bernard and Luc Vinet
for spotting some errors in the earlier versions of this paper. 
The first author thanks Tullio Ceccherini-Silberstein, 
Fabio Scarabotti and Filippo Tolli, for their invitation to visit 
Rome for collaborations, in particular, for telling 
him the contents of their paper \cite{CST2006} that was 
found to be very much relevant 
to the topic, as is seen in the present paper. 
The main part of the 
present work progressed while the first author was 
visiting Rome in February 2023.
The research of the second author is supported by JSPS, Japan KAKENHI Grant Number JP20K03623.
The third author thanks Satoshi Tsujimoto for informing the work by \cite{bi}.
The third author also thanks Kyoto University and Shanghai Jiao Tong University.
The research of the fourth author is supported by National Natural Science Foundation of China No.~11801353.

\bibliographystyle{abbrv}
\bibliography{ASbib}

\end{document}